\documentclass[a4paper,reqno]{amsart}
\usepackage{fixltx2e}
\usepackage[T1]{fontenc}
\usepackage{amssymb,amsthm,amsmath}
\usepackage[british]{babel}
\usepackage[all]{xy}
\usepackage{calrsfs}
\usepackage{graphicx}
\usepackage{hyperref}
\usepackage{mathbbol}
\usepackage{mathabx}
\usepackage{xcolor}
\usepackage{color}
\usepackage{booktabs}

\theoremstyle{plain}
\newtheorem{theorem}[subsection]{Theorem}
\newtheorem{lemma}[subsection]{Lemma}
\newtheorem{proposition}[subsection]{Proposition}
\newtheorem{corollary}[subsection]{Corollary}

\theoremstyle{definition}
\newtheorem{definition}[subsection]{Definition}
\newtheorem{example}[subsection]{Example}
\newtheorem{examples}[subsection]{Examples}
\newtheorem{remark}[subsection]{Remark}

\newenvironment{tfae}
{
\begin{enumerate}}
{\end{enumerate}}

\newcommand{\links}{\lgroup}
\newcommand{\rechts}{\rgroup}

\newcommand{\defn}{\textbf}

\newcommand{\noproof}{\hfill \qed}

\renewcommand{\square}{\raisebox{.4mm}{\,\ensuremath{\boxvoid}}}

\DeclareMathOperator{\cod}{cod} 
\DeclareMathOperator{\End}{End} 

\DeclareMathOperator{\Eq}{Eq}

\newcommand{\C}{\ensuremath{\mathbb{C}}}
\newcommand{\N}{\ensuremath{\mathbb{N}}}

\newcommand{\s}{\ensuremath{\mathcal{S}}}
\newcommand{\p}{\ensuremath{\mathcal{P}}}
\newcommand{\m}{\ensuremath{\mathcal{M}}}

\newcommand{\Ab}{\ensuremath{\mathsf{Ab}}}
\newcommand{\Cat}{\ensuremath{\mathsf{Cat}}}
\newcommand{\CMon}{\ensuremath{\mathsf{CMon}}}
\newcommand{\GMon}{\ensuremath{\mathsf{GMon}}}
\newcommand{\Gp}{\ensuremath{\mathsf{Gp}}}
\newcommand{\Lie}{\ensuremath{\mathsf{Lie}}}
\newcommand{\Mon}{\ensuremath{\mathsf{Mon}}}
\newcommand{\Pt}{\ensuremath{\mathsf{Pt}}}
\newcommand{\Rng}{\ensuremath{\mathsf{Rng}}}
\newcommand{\SRng}{\ensuremath{\mathsf{SRng}}}
\newcommand{\Sub}{\ensuremath{\mathsf{Sub}}}

\newcommand{\M}{\ensuremath{\mathsf{M}}}
\renewcommand{\P}{\ensuremath{\mathsf{P}}}
\renewcommand{\S}{\ensuremath{\mathsf{S}}}
\newcommand{\SU}{\ensuremath{\mathsf{SU}}}
\newcommand{\U}{\ensuremath{\mathsf{U}}}

\def\pullback{
 \ar@{-}[]+R+<6pt,-1pt>;[]+RD+<6pt,-6pt>%
 \ar@{-}[]+D+<1pt,-6pt>;[]+RD+<6pt,-6pt>}

\def\halfsplitpullback{%
 \ar@{-}[]+R+<6pt,-.5ex>;[]+RD+<6pt,-6pt>%
 \ar@{-}[]+D+<1pt,-6pt>;[]+RD+<6pt,-6pt>}

\def\ophalfsplitpullback{%
 \ar@{-}[]+R+<6pt,-1pt>;[]+RD+<6pt,-6pt>%
 \ar@{-}[]+D+<.5ex,-6pt>;[]+RD+<6pt,-6pt>}

\def\splitsplitpullback{%
 \ar@{-}[]+R+<6pt,-.5ex>;[]+RD+<6pt,-6pt>%
 \ar@{-}[]+D+<.5ex,-6pt>;[]+RD+<6pt,-6pt>}

\def\cubepullback{%
 \ar@{-}[]+RD+<6pt,-9pt>;[]+RDD+<6pt,-20pt>%
 \ar@{-}[]+D+<.5ex,-8pt>;[]+RDD+<6pt,-20pt>}

\def\bottomcubepullback{%
 \ar@{-}[]+R+<7pt,0pt>;[]+RD+<18pt,-6pt>%
 \ar@{-}[]+D+<12pt,-6pt>;[]+RDD+<18pt,-6pt>}
 
\newdir{>>}{{}*!/3.5pt/:(1,-.2)@^{>}*!/3.5pt/:(1,+.2)@_{>}*!/7pt/:(1,-.2)@^{>}*!/7pt/:(1,+.2)@_{>}}
\newdir{ >>}{{}*!/8pt/@{|}*!/3.5pt/:(1,-.2)@^{>}*!/3.5pt/:(1,+.2)@_{>}}
\newdir{ |>}{{}*!/-3.5pt/@{|}*!/-8pt/:(1,-.2)@^{>}*!/-8pt/:(1,+.2)@_{>}}
\newdir{ >}{{}*!/-8pt/@{>}}
\newdir{>}{{}*:(1,-.2)@^{>}*:(1,+.2)@_{>}}
\newdir{<}{{}*:(1,+.2)@^{<}*:(1,-.2)@_{<}}

\begin{document}

\title[Two characterisations of groups amongst monoids]{Two characterisations\\ of groups amongst monoids}

\author{Andrea Montoli}
\address[Andrea Montoli]{Dipartimento di Matematica ``Federigo Enriques'', Universit\`{a} degli Studi di Milano, Via
Saldini 50, 20133 Milano, Italy\newline and\newline CMUC,
Department of Mathematics, University of Coimbra, 3001--501
Coimbra, Portugal}
\thanks{The first author acknowledges support by the Programma per Giovani Ricercatori ``Rita Levi Montalcini'', funded by the Italian government through MIUR}
\email{andrea.montoli@unimi.it}

\author{Diana Rodelo}
\address[Diana Rodelo]{CMUC, Department of Mathematics, University of Coimbra,
3001--501 Coimbra, Portugal\newline and\newline Departamento de
Matem\'atica, Faculdade de Ci\^{e}ncias e Tecnologia, Universidade
do Algarve, Campus de Gambelas, 8005--139 Faro, Portugal}
\thanks{The first two authors acknowledge partial financial assistance by Centro de Matem\'{a}tica da
Universidade de Coimbra---UID/MAT/00324/2013, funded by the
Portuguese Government through FCT/MCTES and co-funded by the
European Regional Development Fund through the Partnership
Agreement PT2020.} \email{drodelo@ualg.pt}

\author{Tim Van~der Linden}
\address[Tim Van~der Linden]{Institut de
Recherche en Math\'ematique et Physique, Universit\'e catholique
de Louvain, che\-min du cyclotron~2 bte~L7.01.02, B--1348
Louvain-la-Neuve, Belgium}
\thanks{The third author is a Research
Associate of the Fonds de la Recherche Scientifique--FNRS}
\email{tim.vanderlinden@uclouvain.be}



\date{\today}

\begin{abstract}
The aim of this paper is to solve a problem proposed by Dominique
Bourn: to provide a categorical-algebraic characterisation of
groups amongst monoids and of rings amongst semirings. In the case
of monoids, our solution is given by the following equivalent
conditions:
\begin{tfae}
\item $G$ is a group; \item $G$ is a \emph{Mal'tsev object}, i.e.,
the category $\Pt_{G}(\Mon)$ of points over $G$ in the category of
monoids is unital; \item $G$ is a \emph{protomodular object},
i.e., all points over $G$ are \emph{stably strong}, which means
that any pullback of such a point along a morphism of monoids
${Y\to G}$ determines a split extension
\[
\xymatrix{ 0 \ar[r] & K \ar@{{ |>}->}[r]^-{k} & X \ar@{-{
>>}}@<-.5ex>[r]_-f & Y \ar@{{ >}->}@<-.5ex>[l]_-s \ar[r] & 0 }
\]
in which $k$ and $s$ are jointly strongly epimorphic.
\end{tfae}
We similarly characterise rings in the category of semirings.

On the way we develop a \emph{local} or \emph{object-wise}
approach to certain important conditions occurring in categorical
algebra. This leads to a basic theory involving what we call
\emph{unital} and \emph{strongly unital} objects,
\emph{subtractive} objects, \emph{Mal'tsev} objects and
\emph{protomodular} objects. We explore some of the connections
between these new notions and give examples and counterexamples.
\end{abstract}

\maketitle

\section{Introduction}
The concept of \emph{abelian object} plays a key role in
categorical algebra. In the study of categories of non-abelian
algebraic structures---such as groups, Lie algebras, loops, rings,
crossed modules, etc.---the ``abelian case'' is usually seen as a
basic starting point, often simpler than the general case, or
sometimes even trivial. Most likely there are known results which
may or may not be extended to the surrounding non-abelian setting.
Part of categorical algebra deals with such generalisation issues,
which tend to become more interesting precisely where this
extension is not straightforward. Abstract commutator theory for
instance, which is about \emph{measuring non-abelianness}, would
not exist without a formal interplay between the abelian and the
non-abelian worlds, enabled by an accurate definition of
abelianness.

Depending on the context, several approaches to such a
conceptualisation exist. Relevant to us are those considered
in~\cite{Borceux-Bourn}; see also~\cite{Huq, Smith, Pedicchio} and
the references in~\cite{Borceux-Bourn}. The easiest is probably to
say that an \defn{abelian object} is an object which admits an
internal abelian group structure. This makes sense as soon as the
surrounding category is \emph{unital}---a condition introduced
in~\cite{Bourn1996}, see below for details---which is a rather
weak additional requirement on a pointed category implying that an
object admits at most one internal abelian group structure. So
that, in this context, ``being abelian'' becomes a property of the
object in question.

The full subcategory of a unital category $\C $ determined by the
abelian objects is denoted $\Ab(\C) $ and called the
\defn{additive core} of $\C $. The category $\Ab(\C) $ is indeed
additive, and if $\C $ is a finitely cocomplete
regular~\cite{Barr} unital category, then~$\Ab(\C) $ is a
reflective~\cite{Borceux-Bourn} subcategory of $\C $. If $\C$ is
moreover Barr exact~\cite{Barr}, then~$\Ab(\C) $ is an abelian
category, and called the \defn{abelian core} of $\C $.

For instance, in the category $\Lie_{K}$ of Lie algebras over a
field $K$, the abelian objects are $K$-vector spaces, equipped
with a trivial (zero) bracket; in the category $\Gp$ of groups,
the abelian objects are the abelian groups, so that
$\Ab(\Gp)=\Ab$; in the category $\Mon$ of monoids, the abelian
objects are abelian groups as well: $\Ab(\Mon)=\Ab$; etc. In all
cases the resulting commutator theory behaves as expected.

\subsection*{Beyond abelianness: weaker conditions}
The concept of an abelian object has been well studied and
understood. For certain applications, however, it is too strong:
the ``abelian case'' may not just be \emph{simple}, it may be
\emph{too simple}. Furthermore, abelianness may ``happen too
easily''. As explained in~\cite{Borceux-Bourn}, the
Eckmann--Hilton argument implies that any internal monoid in a
unital category is automatically a \emph{commutative} object. For
instance, in the category of monoids any internal monoid is
commutative, so that in particular an internal group is always
abelian: $\Gp(\Mon)=\Ab$. Amongst other things, this fact is well
known to account for the abelianness of the higher homotopy
groups.

If we want to capture groups amongst monoids, avoiding abelianness
turns out to be especially difficult. One possibility would be to
consider gregarious objects~\cite{Borceux-Bourn}, because the
``equation''
\begin{center}
commutative + gregarious = abelian
\end{center}
holds in any unital category. But this notion happens to be too
weak, since examples were found of gregarious monoids which are
not groups. On the other hand, as explained above, the concept of
an internal group is too strong, since it gives us abelian groups.
Whence the subject of our present paper: to find out how to
\begin{center}
characterise \emph{non-abelian} groups inside the category of
monoids
\end{center}
in categorical-algebraic terms. That is to say, is there some
weaker concept than that of an abelian object which, when
considered in $\Mon$, gives the category~$\Gp$?

This question took quite a long time to be answered. As explained
in~\cite{SchreierBook, BM-FMS2}, the study of monoid actions,
where an \defn{action} of a monoid $B$ on a monoid $X$ is a monoid
homomorphism $B \to \End(X)$ from~$B$ to the monoid of
endomorphisms of~$X$, provided a first solution to this problem: a
monoid $B$ is a group if and only if all split epimorphisms with
codomain~$B$ correspond to monoid actions of~$B$. However, this
solution is not entirely satisfactory, since it makes use of
features which are typical for the category of monoids, and thus
cannot be exported to other categories.

Another approach to this particular question is to consider the
concept of $\s$-protomodularity~\cite{SchreierBook, S-proto,
Bourn2014}, which allows to single out a
protomodular~\cite{Bourn1991} subcategory $\s(\C)$ of a given
category~$\C$, depending on the choice of a convenient class $\s$
of points in $\C$---see below for details. Unlike the category of
monoids, the category of groups is protomodular. And indeed, when
$\C=\Mon$, the class $\s$ of so-called \emph{Schreier
points}~\cite{BM-FMS} does characterise groups in the sense that
$\s(\Mon)=\Gp$. A~similar characterisation is obtained through the
notion of $\s$-Mal'tsev categories~\cite{Bourn2014}. However, these
characterisations are ``relative'', in the sense that they depend on
the choice of a class $\s$. Moreover, the definition of the class
$\s$ of Schreier points is ad-hoc, given that it again crucially
depends on $\C$ being the category of monoids. So the problem is
somehow shifted to another level.

The approach proposed in our present paper is different because it
is \emph{local} and \emph{absolute}, rather than \emph{global} and
\emph{relative}. ``Local'' here means that we consider conditions
defined object by object: \emph{protomodular} objects,
\emph{Mal'tsev} objects, \emph{(strongly) unital} objects and
\emph{subtractive} objects. While $\s$-protomodularity deals with
the protomodular subcategory $\s(\C)$ as a whole. ``Absolute''
means that there is no class $\s$ for the definitions to depend
on.

More precisely, we show in Theorem~\ref{groups = protomodular
monoids} that the notions of a protomodular object and a Mal'tsev
object give the desired characterisation of groups amongst
monoids---whence the title of our paper. Moreover, we find
suitable classes of points which allow us to establish the link
between our absolute approach and the relative approach of
$\s$-protomodularity and the $\s$-Mal'tsev condition
(Proposition~\ref{proto objs=proto core} and
Proposition~\ref{Mal'tsev objs = Mal'tsev core}).

The following table gives an overview of the classes of objects we
consider, and what they amount to in the category of monoids
$\Mon$ and in the category of semi\-rings $\SRng$. Here $\GMon$
denotes the category of gregarious monoids mentioned above.

\begin{table}[h!]
\caption{Special objects in the categories $\Mon$ and $\SRng$}
\begin{tabular}{cccccc}
\toprule \txt{all\\ objects} & \txt{unital\\ objects} &
\txt{subtractive\\ objects} & \txt{strongly unital\\ objects} &
\txt{Mal'tsev\\ objects} & \txt{protomodular\\objects}\\\midrule
$\C$ & $\U(\C)$ & $\S(\C)$ & $\SU(\C)$ & $\M(\C)$ & $\P(\C)$
\\\midrule
$\Mon$ & $\Mon$ & $\GMon$ & $\GMon$ & $\Gp$ & $\Gp$\\
$\SRng$ & $\SRng$ & $\Rng$ & $\Rng$ & $\Rng$ & $\Rng$\\
\bottomrule
\end{tabular}
\label{overview}
\end{table}

In function of the category $\C$ it is possible to separate all
classes of special objects occurring in Table~\ref{overview}.
Indeed, a given category is unital, say, precisely when all of its
objects are unital; while there exist examples of unital
categories which are not subtractive, Mal'tsev categories which
are not protomodular, and so on.

The present paper is the starting point of an exploration of this
new object-wise approach, which is being further developed in
ongoing work. For instance, the article~\cite{GM-ACS} provides a
simple direct proof of a result which implies our
Theorem~\ref{groups = protomodular monoids}, and in~\cite{GM-VdL1}
cocommutative Hopf algebras over an algebraically closed field are
characterised as the protomodular objects in the category of
cocommutative bialgebras.

\subsection*{Example: protomodular objects}
Let us, as an example of the kind of techniques we use, briefly
sketch the definition of a protomodular object. Given an object
$B$, a \defn{point over $B$} is a pair of morphisms $(f\colon
{A\to B},s\colon{B\to A})$ such that $fs=1_{B}$. A~category with
finite limits is said to be
\defn{protomodular}~\cite{Bourn1991,Borceux-Bourn} when for every
pullback
\[
\vcenter{\xymatrix@!0@=5em{ C\times_{B}A \ophalfsplitpullback
\ar[r]^-{\pi_A} \ar@<-.5ex>[d]_-{\pi_C} & A \ar@<-.5ex>[d]_-f
 \\
C \ar@<-.5ex>[u] \ar[r]_-g & B \ar@<-.5ex>[u]_-s }}
\]
of a point $(f,s)$ over $B$ along some morphism $g$ with codomain
$B$, the morphisms~$\pi_A$ and $s$ are \defn{jointly strongly
epimorphic}: they do not both factor through a given proper
subobject of $A$. In a pointed context, this condition is
equivalent to the validity of the \emph{split short five
lemma}~\cite{Bourn1991}. This observation gave rise to the notion
of a \defn{semi-abelian} category---a pointed, Barr exact,
protomodular category with finite
coproducts~\cite{Janelidze-Marki-Tholen}---which plays a
fundamental role in the development of a categorical-algebraic
approach to homological algebra for non-abelian structures; see
for
instance~\cite{Bourn-Janelidze:Torsors,EGVdL,Butterflies,CMM1,RVdL2}.

A point $(f,s)$ satisfying the condition mentioned above (that
$\pi_A$ and $s$ are jointly strongly epimorphic) is called a
\defn{strong point}. When also all of its pullbacks satisfy this
condition, it is called a \defn{stably strong point}. We shall say
that $B$ is a \defn{protomodular object} when all points over $B$
 are stably strong points. Writing $\P(\C)$ for the
full subcategory of $\C$ determined by the protomodular objects,
we clearly have that $\P(\C)=\C$ if and only if $\C$ is a
protomodular category. In fact, $\P(\C)$ is \emph{always} a
protomodular category, as soon as it is closed under finite limits
in~$\C$. We study some of its basic properties in
Section~\ref{Protomodular objects}, where we also prove one of our
main results: if $\C$ is the category of monoids, then $\P(\C)$ is
the category of groups (Theorem~\ref{groups = protomodular
monoids}). This is one of two answers to the question we set out
to study, the other being a characterisation of groups amongst
monoids as the so-called \emph{Mal'tsev objects} (essentially
Theorem~\ref{Mal'tsev monoids are groups}).

\subsection*{Structure of the text}
Since the concept of a (stably) strong point plays a key role in
our work, we recall its definition and discuss some of its basic
properties in Section~\ref{SSP}. Section~\ref{section S-Mal'tsev
and S-protomodular} recalls the definitions of $\s$-Mal'tsev and
$\s$-protomodular categories in full detail.

In Section~\ref{SUO} we introduce the concept of \emph{strongly
unital} object. We show that these coincide with the
\emph{gregarious} objects when the surrounding category is
regular. We prove stability properties and characterise rings
amongst semirings as the strongly unital objects
(Theorem~\ref{SU(SRng)=Rng}).

Section~\ref{USO} is devoted to the concepts of \emph{unital} and
\emph{subtractive} object. Our main result here is
Proposition~\ref{SU=SU} which, mimicking Proposition~3
in~\cite{ZJanelidze-Subtractive}, says that an object of a pointed
regular category is strongly unital if and only if it is unital
and subtractive.

In Section~\ref{MO} we introduce \emph{Mal'tsev} objects and prove
that any Mal'tsev object in the category of monoids is a group
(Theorem~\ref{Mal'tsev monoids are groups}).
Section~\ref{Protomodular objects} treats the concept of a
\emph{protomodular} object. Here we prove our paper's main result,
Theorem~\ref{groups = protomodular monoids}: a monoid is a group
if and only if it is a protomodular object, and if and only if it
is a Mal'tsev object. We also explain in which sense the full
subcategory determined by the protomodular objects is a
protomodular core~\cite{S-proto}.

\section{Stably strong points}\label{SSP}

We start by recalling some notions that occur frequently in
categorical algebra, focusing on the concept of a \emph{strong
point}.

\subsection{Jointly strongly epimorphic pairs}
A cospan $(r\colon C\to A, s\colon B\to A)$ in a category $\C$ is
said to be \defn{jointly extremally epimorphic} when it does not
factor through a monomorphism, which means that for any
commutative diagram where $m$ is a monomorphism
\[
\xymatrix@!0@=4em{ & M \ar@{ >->}[d]^- m \\
 C \ar[r]_-r \ar[ur] & A & B, \ar[l]^-s \ar[ul]}
\]
the monomorphism $m$ is necessarily an isomorphism. If $\C$ is
finitely complete, then it is easy to see that the pair $(r,s)$ is
jointly epimorphic. In fact, in a finitely complete category the
notions of extremal epimorphism and strong epimorphism coincide.
Therefore, we usually refer to the pair $(r,s)$ as being
\defn{jointly strongly epimorphic}. Recall that, if $\C$ is
moreover a regular category~\cite{Barr}, then extremal
epimorphisms and strong epimorphisms coincide with the regular
epimorphisms.

\subsection{The fibration of points}
A \defn{point} $(f\colon{A\to B},s\colon{B\to A})$ in $\C$ is a
split epimorphism $f$ with a chosen splitting $s$. Considering a
point as a diagram in~$\C$, we obtain the category of points in
$\C$, denoted $\Pt(\C) $: morphisms between points are pairs $(x,
y) \colon(f,s)\to (f',s')$ of morphisms in $\C$ making the diagram
\[
\xymatrix@!0@=4em{B \ar[r]^-{s} \ar[d]_y & A \ar[r]^-{f} \ar[d]^-{x} & B \ar[d]^y \\
B' \ar[r]_{s'} & A' \ar[r]_-{f'} & B'}
\]
commute. If $\C$ has pullbacks of split epimorphisms, then the
forgetful functor $\cod \colon {\Pt(\C) \to \C}$, which associates
with every split epimorphism its codomain, is a fibration, usually
called the
\defn{fibration of points}~\cite{Bourn1991}. Given an object $B$
of~$\C$, we denote the fibre over $B$ by $\Pt_B(\C)$. An object in
this category is a point with codomain $B$, and a morphism is of
the form~$(x, 1_B)$.

\subsection{Strong points}
We now assume $\C$ to be a finitely complete category.

\begin{definition}
We say that a point $(f\colon{A\to B},s\colon{B\to A})$ is a
\defn{strong point} when for every pullback
\begin{equation}
\label{strong point diagram} \vcenter{\xymatrix@!0@=5em{
C\times_{B}A \ophalfsplitpullback \ar[r]^-{\pi_A}
\ar@<-.5ex>[d]_-{\pi_C} & A \ar@<-.5ex>[d]_-f
 \\
C \ar@<-.5ex>[u]_(.4){\langle 1_{C}, sg \rangle} \ar[r]_-g & B
\ar@<-.5ex>[u]_-s }}
\end{equation}
along any morphism $g \colon {C \to B}$, the pair $(\pi_A, s)$ is
jointly strongly epimorphic.
\end{definition}

Strong points were already considered
in~\cite{MartinsMontoliSobral2}, under the name of \emph{regular
points} (in a regular context), and independently
in~\cite{Bourn-monad}, under the name of \emph{strongly split
epimorphisms}.

Many algebraic categories have been characterised in terms of
properties of strong points (see~\cite{Bourn1996, Borceux-Bourn}),
some of which we recall throughout the text. For instance, by
definition, a finitely complete category is
\defn{protomodular}~\cite{Bourn1991} precisely when all points in
it are strong. For a pointed category, this condition is
equivalent to the validity of the split short five
lemma~\cite{Bourn1991}. Examples of protomodular categories are
the categories of groups, of rings, of Lie algebras (over a
commutative ring with unit) and, more generally, every
\emph{variety of $\Omega$-groups} in the sense of Higgins
\cite{Higgins}. Protomodularity is also a key ingredient in the
definition of a \emph{semi-abelian
category}~\cite{Janelidze-Marki-Tholen}.

On the other hand, in the category of sets, a point $(f,s)$ is
strong if and only if $f$ is an isomorphism. To see this, it
suffices to pull it back along the unique morphism from the empty
set $\varnothing$.

\subsection{Pointed categories}
In a pointed category, we denote the kernel of a morphism $f$ by
$\ker(f)$. In the pointed case, the notion of strong point
mentioned above coincides with the one considered
in~\cite{MRVdL4}:

\begin{proposition}
Let $\C$ be a pointed finitely complete category.
\begin{enumerate}
\item A point $(f,s)$ in $\C$ is strong if and only if the pair
$(\ker (f),s)$ is jointly strongly epimorphic. \item Any split
epimorphism $f$ in a strong point $(f,s)$ is a normal epimorphism.
\end{enumerate}
\end{proposition}
\begin{proof}
(1) If $(f,s)$ is a strong point, then $(\ker (f),s)$ is jointly
strongly epimorphic: to see this, it suffices to take the pullback
of $f$ along the unique morphism with domain the zero object.
Conversely, if we take an arbitrary pullback as in~\eqref{strong
point diagram}, then $\ker (f)=\pi_A \langle 0, \ker (f) \rangle$.
We conclude that $(\pi_A,s)$ is jointly strongly epimorphic
because $(\ker (f),s)$ is.

(2) Since $(f,s)$ is a strong point, the pair $(\ker(f),s)$ is
jointly strongly epimorphic; thus it is jointly epimorphic. It
easily follows that~$f$ is the cokernel of its kernel~$\ker(f)$.
\end{proof}

In a pointed finitely complete context, asking that certain
product projections are strong points gives rise to the notions of
a unital and of a strongly unital category. In fact, when for all
objects $X$, $Y$ in $\C$ the point
\[
 (\pi_{X}\colon {X\times Y\to X},\quad \langle 1_{X},0 \rangle\colon X\to X\times Y)
\]
is strong, $\C$ is said to be a \defn{unital}
category~\cite{Bourn1996}. The category $\C$ is called
\defn{strongly unital} (\cite{Bourn1996}, see also Definition 1.8.3 and Theorem 1.8.15 in \cite{Borceux-Bourn}) when for every object~$X$ in $\C$
the point
\[
 (\pi_{1}\colon {X\times X\to X},\quad \Delta_{X}=\langle 1_{X},1_{X}\rangle\colon X\to X\times X)
\]
is strong. Observe that we could equivalently ask the point
$(\pi_2, \Delta_X)$ to be strong. It is well known that every
strongly unital category is necessarily unital~\cite[Proposition
1.8.4]{Borceux-Bourn}.

\begin{example}\label{Examples unital}
As shown in~\cite[Theorem 1.2.15]{Borceux-Bourn}, a variety in the
sense of universal algebra is a unital category if and only if it
is a \defn{J\'{o}nsson--Tarski variety}. This means that the
corresponding theory contains a unique constant~$0$ and a binary
operation $+$ subject to the equations $0 + x = x= x + 0$.
\end{example}

In particular, the categories of monoids and of semirings are
unital. Moreover, every pointed protomodular category is strongly
unital.

\subsection{Stably strong points}
We are especially interested in those points for which the
property of being strong is pullback-stable.

\begin{definition}
We say that a point $(f,s)$ is
\defn{stably strong} if every pullback of it along any morphism is
a strong point. More explicitly, for any morphism~$g$, the point
$(\pi_C, \langle 1_C, sg \rangle)$ in Diagram~\eqref{strong point
diagram} is strong.
\end{definition}

Note that a stably strong point is always strong (it suffices to
pull it back along the identity morphism) and that the collection
of stably strong points determines a subfibration of the fibration
of points. In a protomodular category, \emph{all} points are
stably strong (since all points are strong). In the category of
sets, all strong points are stably strong (since isomorphisms are
preserved by pullbacks). Nevertheless, in a finitely complete
category not all strong points are stably strong as can be seen in
the following examples.

\begin{example}
Let $\C$ be any pointed non-unital category. (For instance, the
category of Hopf algebras over a field is such~\cite{GM-VdL1}.)
Necessarily then, certain product inclusions are not jointly
strongly epimorphic. Let $(\pi_{X},\langle 1_{X},0 \rangle)\colon
{X\times Y\leftrightarrows X}$ be a product projection which is
not a strong point. It is a pullback of the point $Y
\leftrightarrows 0$, which is obviously strong---but not stably
strong.
\end{example}

\begin{example}\label{only strong}
 A variety of universal algebras is said to be \defn{subtractive}~\cite{Ursini3} when the corresponding
 theory contains a unique constant $0$ and a binary operation $s$, called a \defn{subtraction}, subject to the
 equations $s(x,0) = x$ and $s(x,x) = 0$. We write $\Sub$ for the subtractive variety of \defn{subtraction algebras}, which are
 triples $(X,s,0)$ where $X$ is a set, $s$ a subtraction on $X$ and $0$ the corresponding constant.

Let $T$ be the subtraction algebra
$$
\begin{array}{c|cc}
 s & 0 & a \\
 \hline
 0 & 0 & 0 \\
 a & a & 0
\end{array}
$$
Then $(\pi_1, \Delta_T)\colon {T\times T\leftrightarrows T}$ is a
strong point, since $(\langle 0,1_{T}\rangle, \Delta_T)$ is a
jointly strongly epimorphic pair of arrows. Indeed,
$(a,0)=(s(a,0),s(a,a))=s((a,a),(0,a))$.

Let $X$ be the subtraction algebra
$$
\begin{array}{c|ccc}
 s & 0 & u & v \\
 \hline
 0 & 0 & 0 & 0\\
 u & u & 0 & 0 \\
 v & v & 0 & 0
\end{array}
$$
and consider the constant map $f\colon X\to T\colon x\mapsto 0$.
The pullback of the point $(\pi_1, \Delta_T)\colon {T\times
T\leftrightarrows T}$ along $f$ gives the point $(\pi_X,\langle
1_X,0\rangle)\colon {X\times T\leftrightarrows X}$.

It is easy to see that this point is not strong: the only way the
pair $(u,a)\in X\times T$ can be written as a difference is
$(u,a)=(s(u,0),s(a,0))=s((u,a),(0,0))$. Alternatively, we can
consider the subalgebra $M=\{(0,0), (0,a),(u,0),(v,0)\}$ of the
product~${X\times T}$. $M$ is strictly smaller than $X\times T$,
since it does not contain the element~$(u,a)$. Note that the
restriction of the subtraction on $X\times T$ to $M$ is given by
$$
\begin{array}{c|cccc}
 s & (0,0) & (0,a) & (u,0) & (v,0) \\
 \hline
 (0,0) & (0,0) & (0,0) & (0,0) & (0,0) \\
 (0,a) & (0,a) & (0,0) & (0,a) & (0,a) \\
 (u,0) & (u,0) & (u,0) & (0,0) & (0,0) \\
 (v,0) & (v,0) & (v,0) & (0,0) & (0,0)
\end{array}
$$
so it does indeed define an operation on $M$. On the other hand,
the two product inclusions $\langle1_{X},0\rangle$ and $\langle
0,1_{T}\rangle$ do factor through $M$.

This allows us to conclude that the point $(\pi_1, \Delta_T)\colon
{T\times T\leftrightarrows T}$ is not stably strong.
\end{example}

\subsection{The regular case}
In the context of regular categories~\cite{Barr}, (stably) strong
points are
\defn{closed under quotients}: this means that in any commutative
diagram
\[
\xymatrix@!0@=4em{
A \ar@<-.5ex>[d]_f \ar@{->>}[r]^\alpha & A' \ar@<-.5ex>[d]_{f'} \\
B \ar@<-.5ex>[u]_s \ar@{->>}[r]_\beta & B', \ar@<-.5ex>[u]_{s'} }
\]
where $\alpha$ and $\beta $ are regular epimorphisms and $(f,s)$
is (stably) strong, also $(f',s')$ is (stably) strong.

\begin{proposition} \label{stably strong points closed under quotients}
In a finitely complete category, strong points are closed under
quotients and stably strong points are closed under retractions.
In a regular category, stably strong points are closed under
quotients.
\end{proposition}
\begin{proof}
Let us first prove that the quotient of a strong point is always
strong. So let $(f,s)$ be a strong point, and consider the diagram
\[
\xymatrix@=4em@!0{ P \cubepullback \ar@{->}[rr]^{\alpha'}
\ar@<-.5ex>[dd] \ar[dr]^-{\pi_{A}} & & P' \cubepullback
\ar@<-.5ex>[dd]|{\hole} \ar[dr]^-{\pi_{A'}} & \\
& A \ar@<-.5ex>[dd]_(.3)f \ar@{->>}[rr]^(.3)\alpha & & A'
\ar@<-.5ex>[dd]_{f'} \\
C \bottomcubepullback \ar[dr]_{g} \ar@<-.5ex>[uu]
\ar@{->}[rr]|(.47){\hole}|(.53){\hole}^(.7){\beta '} & & C'
\ar@<-.5ex>[uu]|{\hole} \ar[dr]_{g'} & \\
& B \ar@<-.5ex>[uu]_(.7)s \ar@{->>}[rr]_\beta & & B',
\ar@<-.5ex>[uu]_-{s'} }
\]
where $P'$ is the pullback of $f'$ along an arbitrary morphism
$g'$, $C$ is the pullback of~$g'$ along $\beta $, and~$P$ is the
pullback of $f$ along $g$. By pullback cancelation, the upper
square is a pullback too. Since $\alpha$ is a regular epimorphism,
we have that $\alpha\pi_{A}$ and $\alpha s$ are jointly strongly
epimorphic. Then it easily follows that $\pi_{A'}$ and $s'$ are
jointly strongly epimorphic, so that the point $(f', s')$ is a
strong point.

If now $(f,s)$ is stably strong, then the point $P
\leftrightarrows C$ is strong. If $\alpha$ and $\beta$ are
retractions, then so are $\alpha'$ and $\beta'$. If $\alpha$ and
$\beta$ are regular epimorphisms in a regular category, then so
are $\alpha'$ and $\beta'$. In both cases, $P' \leftrightarrows
C'$ is strong as a quotient of $P \leftrightarrows C$. Hence
$(f',s')$ is stably strong.
\end{proof}

As a consequence, in a regular category, a point $(f,s)$ is stably
strong if and only if the point $(\pi_1,\langle 1_A, sf \rangle)$
induced by its kernel pair is stably strong. Equivalently one
could consider the point $(\pi_2,\langle sf, 1_A \rangle)$.

Certain pushouts involving strong points satisfy a stronger
property. Recall from~\cite{Bourn2003} that a \defn{regular
pushout} in a regular category is a commutative square of regular
epimorphisms
\[
\vcenter{\xymatrix@!0@=4em{A' \ar@{->>}[d]_-{f'} \ar@{->>}[r]^-{\alpha} & A \ar@{->>}[d]^-f\\
B' \ar@{->>}[r]_-{\beta} & B}}
\]
where also the comparison arrow $\langle f',\alpha\rangle\colon
A'\to B'\times_{B}A$ is a regular epimorphism. Every regular
pushout is a pushout.

A \defn{double split epimorphism} in a category $\C$ is a point in
the category of points in $\C$, so a commutative diagram
\begin{equation}
\label{double split extension} \vcenter{ \xymatrix@!0@=4em{ D
\ar@<-.5ex>[d]_{g'} \ar@<-.5ex>[r]_{f'} & C \ar@<-.5ex>[d]_g
\ar@<-.5ex>[l]_{s'} \\
A \ar@<-.5ex>[u]_{t'} \ar@<-.5ex>[r]_f & B \ar@<-.5ex>[l]_s
\ar@<-.5ex>[u]_t }}
\end{equation}
where the four ``obvious'' squares commute.

\begin{lemma}\label{Lemma Double}
In a regular category, every double split epimorphism as in
\eqref{double split extension}, in which $(g,t)$ is a stably
strong point, is a regular pushout.
\end{lemma}
\begin{proof}
Take the pullback $A\times_{B}C$ of $f$ and $g$, consider the
comparison morphism $\langle g',f'\rangle\colon {D\to
A\times_{B}C}$ and factor it as a regular epimorphism $e\colon
{D\to M}$ followed by a monomorphism $m\colon {M\to
A\times_{B}C}$. Since $(g,t)$ is a stably strong point, its
pullback $(\pi_{A},\langle1_{A},tf\rangle)$ in the diagram
\[
\xymatrix@!0@C=5em@R=4em{ C \ophalfsplitpullback \ar@<-.5ex>[d]_-g
 \ar[r]^-{\langle sg, 1_{C} \rangle} & A\times_{B}C \ophalfsplitpullback \ar[r]^-{\pi_C}
\ar@<-.5ex>[d]_-{\pi_A} & C \ar@<-.5ex>[d]_-g
 \\
B \ar@<-.5ex>[u]_(.4)t \ar[r]_-{s} & A \ar@<-.5ex>[u]_(.4){\langle
1_{A},tf \rangle} \ar[r]_-f & B \ar@<-.5ex>[u]_-t }
\]
is a strong point. As a consequence, the pair $(\langle
sg,1_{C}\rangle, \langle1_{A},tf\rangle)$ is jointly strongly
epimorphic. They both factor through the monomorphism $m$ as in
the diagram
\[
\xymatrix@!0@R=4em@C=5em{ & M \ar@{{ >}->}[d]^-{m} & \\
C \ar[ur]^-{es'} \ar[r]_-{\langle sg, 1_{C} \rangle} &
A\times_{B}C & A, \ar[l]^-{\langle 1_{A},tf \rangle} \ar[ul]_{et'}
}
\]
so that $m$ is an isomorphism.
\end{proof}

\begin{lemma}\label{Bourn Lemma}
In a regular category, consider a commutative square of regular
epimorphisms with horizontal kernel pairs
\begin{equation*}\label{SpecialRG}
\vcenter{\xymatrix@!0@=4em{\Eq(g) \ar@<-1ex>[r] \ar@{->>}[d]_-{f''} \ar@<1ex>[r] & A' \ar[l] \ar@{->>}[d]^-{f'} \ar@{->>}[r]^-{g} & A \ar@{->>}[d]^-f\\
\Eq(h) \ar@<-1ex>[r] \ar@<1ex>[r]
 & B' \ar[l] \ar@{->>}[r]_-{h} & B. }}
\end{equation*}
If any of the commutative squares on the left is a regular pushout
(and so, in particular, $f''$ is a regular epimorphism), then the
square on the right is also a regular pushout.
\end{lemma}
\begin{proof}
The proof is essentially the same as the one of Proposition 3.2 in
\cite{Bourn2003}.
\end{proof}

\begin{proposition}
In a regular category, every regular epimorphism of points
\[
\xymatrix@!0@=3em{ D \ar@<-.5ex>[d] \ar@{->>}[r] & C \ar@<-.5ex>[d] \\
A \ar@<-.5ex>[u] \ar@{->>}[r] & B, \ar@<-.5ex>[u] }
\]
where the point on the left (and hence also the one on the right)
is stably strong, is a regular pushout.
\end{proposition}
\begin{proof}
This follows immediately from Lemma~\ref{Lemma Double} and
Lemma~\ref{Bourn Lemma}.
\end{proof}

\section{$\s$-Mal'tsev and $\s$-protomodular categories}
\label{section S-Mal'tsev and S-protomodular}

As mentioned in Section~\ref{SSP}, a finitely complete category
$\C$ in which all points are (stably) strong defines a
protomodular category. If such an ``absolute'' property fails, one
may think of protomodularity in ``relative'' terms, i.e., with
respect to a class $\s$ of stably strong points. We also recall
the absolute and relative notions for the Mal'tsev context.

Recall that a finitely complete category $\C$ is called a
\defn{Mal'tsev category}~\cite{CLP, CPP} when every internal
reflexive relation in $\C$ is automatically symmetric or,
equivalently, transitive; thus an equivalence relation.
Protomodular categories are always Mal'tsev
categories~\cite{Bourn1996}. If $\C$ is a regular category, then
$\C$ is a Mal'tsev category when the composition of any pair of
(effective) equivalence relations $R$ and $S$ on a same object
commutes: $RS=SR$~\cite{CLP, Carboni-Kelly-Pedicchio}. Moreover,
Mal'tsev categories admit a well-known characterisation through
the fibration of points:

\begin{proposition}\cite[Proposition 10]{Bourn1996}\label{Mal'tsev via fibres}
A finitely complete category $\C$ is a Mal'\-tsev category if and
only if every fibre $\Pt_Y(\C)$ is (strongly) unital.\noproof
\end{proposition}

The condition that $\Pt_Y(\C)$ is unital means that, for
every pullback of split epimorphisms
\begin{equation}
\label{pb of split epis} \vcenter{\xymatrix@!0@=5em{ A\times_{Y}C
\splitsplitpullback \ar@<-.5ex>[d]_{\pi_A}
\ar@<-.5ex>[r]_(.7){\pi_C} & C \ar@<-.5ex>[d]_g
\ar@<-.5ex>[l]_-{\langle sg,1_C \rangle} \\
A \ar@<-.5ex>[u]_(.4){\langle 1_A,tf \rangle} \ar@<-.5ex>[r]_f & Y
\ar@<-.5ex>[l]_s \ar@<-.5ex>[u]_t }}
\end{equation}
(which is a binary product in $\Pt_Y(\C)$), the morphisms $\langle
1_{A}, tf \rangle$ and $\langle sg, 1_{C} \rangle$ are jointly
strongly epimorphic.

Let $\C$ be a finitely complete category, and $\s$ a class of
points which is stable under pullbacks along any morphism.

\begin{definition} \label{S-Mal'tsev and S-protomodular categories} Suppose that the full subcategory of
$\Pt(\C)$ whose objects are the points in $\s$ is closed in
$\Pt(\C)$ under finite limits. The category $\C$ is said to be:
\begin{enumerate}
\item \defn{$\s$-Mal'tsev}~\cite{Bourn2014} if, for every pullback
of split epimorphisms~\eqref{pb of split epis} where the point
$(f,s)$ is in the class $\s$, the morphisms $\langle 1_{A}, tf
\rangle$ and $\langle sg, 1_{C} \rangle$ are jointly strongly
epimorphic; \item \defn{$\s$-protomodular}~\cite{SchreierBook,
S-proto, Bourn2014} if every point in $\s$ is strong.
\end{enumerate}
\end{definition}

The notion of $\s$-protomodular category was introduced to
describe, in categorical terms, some convenient properties of
\emph{Schreier split epimorphisms} of monoids and of semirings.
Such split epimorphisms were introduced in~\cite{BM-FMS} as those
points which correspond to classical monoid actions and, more
generally, to actions in every category of \emph{monoids with
operations}, via a semidirect product construction.

In~\cite{SchreierBook, BM-FMS2} it was shown that, for Schreier
split epimorphisms, relative versions of some properties of all
split epimorphisms in a protomodular category hold, like for
instance the \emph{split short five lemma}.

In~\cite{S-proto} it is proved that every category of monoids with
operations, equipped with the class $\s$ of Schreier points, is
$\s$-protomodular, and hence an $\s$-Mal'tsev category. Indeed, as
shown in~\cite{S-proto, Bourn2014}, every $\s$-protomodular
category is an $\s$-Mal'tsev category. Later, in
\cite{MartinsMontoliSH} it was proved that every
J\'{o}nsson--Tarski variety is an $\s$-proto\-modular category
with respect to the class $\s$ of Schreier points.
A~(non-absolute) example of an $\s$-Mal'tsev category which is not
$\s$-protomodular, given in \cite{Bourn-quandles}, is the category
of quandles.

The following definition first appeared in~\cite[Definition
6.1]{S-proto} for pointed $\s$-protomodular categories, then it
was extended in~\cite{Bourn2014} to $\s$-Mal'tsev categories.

\begin{definition}
Let $\C$ be a finitely complete category and $\s$ a class of
points which is stable under pullbacks along any morphism. An
object $X$ in $\C$ is
\defn{$\s$-special} if the point
\begin{equation*}
(\pi_{1}\colon{X\times X\to X},\quad \Delta_{X}=\langle
1_{X},1_{X} \rangle\colon {X\to X\times X})
\end{equation*}
belongs to $\s$ or, equivalently, if the point $(\pi_2, \Delta_X)$
belongs to $\s$. We write $\s(\C)$ for the full subcategory of
$\C$ determined by the $\s$-special objects.
\end{definition}

According to Proposition 6.2 in~\cite{S-proto} and its
generalisation~\cite[Proposition 4.3]{Bourn2014} to $\s$-Mal'tsev
categories, if $\C$ is an $\s$-Mal'tsev category, then the
subcategory $\s(\C)$ of $\s$-special objects of $\C$ is a Mal'tsev
category, called the
\defn{Mal'tsev core} of $\C$ relatively to the class $\s$. When
$\C$ is $\s$-protomodular, $\s(\C)$ is a protomodular category,
called the \defn{protomodular core} of $\C$ relatively to the
class $\s$.

Proposition 6.4 in~\cite{S-proto} shows that the protomodular core
of the category $\Mon$ of monoids relatively to the class $\s$ of
Schreier points is the category $\Gp$ of groups; similarly, the
protomodular core of the category $\SRng$ of semirings is the
category $\Rng$ of rings, also with respect to the class of
Schreier points.

Our main problem in this work is to obtain a categorical-algebraic
characterisation of groups amongst monoids, and of rings amongst
semirings. Based on the previous results, one direction is to look
for a suitable class $\s$ of stably strong points in a general
finitely complete category $\C$ such that the full subcategory
$\s(\C)$ of $\s$-special objects gives the category of groups when
$\C$ is the category of monoids and gives the category of rings
when $\C$ is the category of semirings: $\s(\Mon)=\Gp$ and
$\s(\SRng)=\Rng$.

We explore different possible classes in the following sections as
well as the outcome for the particular cases of monoids and
semirings. A first ``obvious'' choice is to consider $\s$ to be
the class of \emph{all} stably strong points in $\C$. Then an
$\s$-special object is precisely what we call a strongly unital
object in the next section. We shall see that the subcategory
$\s(\C)$ of $\s$-special objects is the protomodular core (namely
$\Rng$) in the case of semirings, but not so in the case of
monoids. Moreover, we propose an alternative ``absolute'' solution
to our main problem, not depending on the choice of a class $\s$
of points, and we compare it with this ``relative'' one.

\section{Strongly unital objects}\label{SUO}

The aim of this section is to introduce the concept of a strongly
unital object. We characterise rings amongst semirings as the
strongly unital objects (Theorem~\ref{SU(SRng)=Rng}). We prove
stability properties for strongly unital objects and show that, in
the regular case, they coincide with the \emph{gregarious} objects
of~\cite{Borceux-Bourn}.

Let $\C$ be a pointed finitely complete category.

\begin{definition} \label{definition SU}
Given an object $Y$ of $\C$, we say that $Y$ is \defn{strongly
unital} if the point
\begin{equation*}
(\pi_{1}\colon{Y\times Y\to Y},\quad \Delta_{Y}=\langle
1_{Y},1_{Y} \rangle\colon {Y\to Y\times Y})
\end{equation*}
is stably strong.
\end{definition}

Note that we could equivalently ask that the point
$(\pi_{2},\Delta_{Y})$ is stably strong. We write $\SU(\C)$ for
the full subcategory of $\C$ determined by the strongly unital
objects.

\begin{remark}\label{Stably strong not Schreier}
An object $Y$ in $\C$ is strongly unital if and only if it is
$\s$-special, when $\s$ is the class of all stably strong points
in $\C$.
\end{remark}

\begin{theorem}
\label{SU(SRng)=Rng} If $\C$ is the category $\SRng$ of semirings,
then $\SU(\C)$ is the category $\Rng$ of rings. In other words, a
semiring $X$ is a ring if and only if the point
\[
(\pi_{1}\colon{X\times X\to X},\quad \Delta_{X}=\langle
1_{X},1_{X} \rangle\colon {X\to X\times X})
\]
is stably strong in $\SRng$.
\end{theorem}
\begin{proof}
If $X$ is a ring, then every point over it is stably strong: by
Proposition~6.1.6 in~\cite{SchreierBook} it is a Schreier point,
and Schreier points of semirings are stably strong by Lemma~6.1.1
combined with Proposition~6.1.8 of~\cite{SchreierBook}. Hence, it
suffices to show that any strongly unital semiring is a ring.
Suppose that the point
\[
(\pi_{1}\colon{X\times X\to X},\quad \Delta_{X}=\langle
1_{X},1_{X} \rangle\colon {X\to X\times X})
\]
is stably strong. Given any element $x \neq 0_X$ of $X$, consider
the pullback of $\pi_1$ along the morphism $x \colon \N \to X$
sending $1$ to $x$:
\[
\xymatrix@!0@C=6em@R=4em{ & X \ar@{{ >}->}[dl]_{\langle 0, 1_{X}
\rangle}
\ar@{{ >}->}[d]^{\langle 0, 1_{X} \rangle} \\
\N \times X \ophalfsplitpullback \ar[r]^{x \times 1_{X}}
\ar@<-.5ex>[d]_{\pi_1} & X \times X
\ar@<-.5ex>[d]_{\pi_1} \\
\N \ar@<-.5ex>[u]_(.4){\langle 1_{\N}, x \rangle} \ar[r]_x & X.
\ar@<-.5ex>[u]_{\langle 1_{X}, 1_{X} \rangle} } \] Consider the
element $(1,0_X) \in \N \times X$. Since the morphisms $\langle
1_{\N}, x \rangle$ and $\langle 0, 1_{X} \rangle$ are jointly
strongly epimorphic, $(1,0_X)$ can be written as the sum of
products of chains of elements of the form $(0, \bar{x})$ and $(n,
nx)$. Using the fact that $0 \in \N$ is absorbing for the
multiplication in $\N$ and that in every semiring the sum is
commutative and the multiplication is distributive with respect to
the sum, we get that $(1,0_X)$ can be written as
\[
(1, 0_X) = (0, y) + (1, x)
\]
for a certain $y \in X$. Then $y+ x = 0_X$ and hence the element
$x$ is invertible for the sum. Thus we see that $X$ is a ring.
\end{proof}

\begin{remark}
Note that, in particular, $\SU(\SRng)=\Rng$ is a protomodular
category, so that $\Rng$ is the protomodular core of $\SRng$ with
respect to the class $\s$ of all stably strong points. As such, it
is necessarily the largest protomodular core of~$\SRng$ induced by
some class $\s$.
\end{remark}

Recall from~\cite{Borceux-Bourn,Bourn2002} that a \defn{split
right punctual span} is a diagram of the form
\begin{equation}\label{srps}
\xymatrix@!0@=4em{ X \ar@<.5ex>[r]^s & Z \ar@<.5ex>[l]^f
\ar@<-.5ex>[r]_g & Y \ar@<-.5ex>[l]_t }
\end{equation}
where $fs = 1_X$, $gt = 1_Y$ and $ft = 0$.

\begin{proposition}\label{SU characterisation}
If $\C$ is a pointed finitely complete category, then the
following conditions are equivalent:
\begin{tfae}
\item $Y$ is a strongly unital object of $\C$; \item for every
morphism $f\colon {X\to Y}$, the point
\[
(\pi_{X}\colon{X\times Y\to X},\quad \langle 1_{X},f \rangle\colon
{X\to X\times Y})
\]
is stably strong; \item for every $f\colon {X\to Y}$, the point
$(\pi_{X}, \langle 1_{X},f \rangle)$ is strong; \item given any
split right punctual span~\eqref{srps}, the map $\langle f, g
\rangle \colon {Z \to X \times Y}$ is a strong epimorphism.
\end{tfae}
\end{proposition}

\begin{proof}
The equivalence between (i), (ii) and (iii) hold since any
pullback of the point $(\pi_1,\Delta_Y)$ is of the form
$(\pi_X,\langle 1_X, f\rangle)$ and any pullback of $(\pi_X,
\langle 1_X,f\rangle)$ is also a pullback of $(\pi_1, \Delta_Y)$.

To prove that (iii) implies (iv), consider a split right punctual
span as in~\eqref{srps}. By assumption, the point $(\pi_{X}\colon
{X\times Y \to X}, \langle 1_X, gs \rangle \colon {X\to X\times
Y})$ is strong. Suppose that $\langle f, g \rangle$ factors
through a monomorphism $m$
\[ \xymatrix@!0@=4em{ X \ar@{=}[dd] \ar@<.5ex>[rr]^s & & Z \ar[dl]_e \ar[dd]^{\langle f, g \rangle} \ar@<.5ex>[ll]^f \ar@<-.5ex>[rr]_g & & Y
\ar@<-.5ex>[ll]_t \ar@{=}[dd] \\
& M \ar@{{ >}->}[dr]^m & & & \\
X \ar@<.5ex>[rr]^{\langle 1_X, gs \rangle} & & X \times Y
\ar@<.5ex>[ll]^{\pi_X} & & Y. \ar[ll]^{\langle 0, 1_Y \rangle} }
\] Both $\langle 1_{X}, gs \rangle$ and $\langle 0, 1_{Y} \rangle$
factor through $m$, indeed $\langle 1_{X}, gs \rangle = mes$ and
$\langle 0, 1_{Y} \rangle = met$. Since $\langle 1_X, gs \rangle$
and $\langle 0, 1_{Y} \rangle$ are jointly strongly epimorphic,
$m$ is an isomorphism.

To prove that (iv) implies (iii), we must show that $\langle 0,
1_{Y} \rangle \colon Y \to X \times Y$ and $\langle 1_{X}, f
\rangle \colon X \to X \times Y$ are jointly strongly epimorphic.
Suppose that they factor through a monomorphism $m = \langle m_1,
m_2 \rangle \colon M \to X \times Y$:
\[
\xymatrix@!0@=4em{ & M \ar@{{ >}->}[d]|{\langle m_1, m_2 \rangle} & \\
X \ar[ur]^a \ar[r]_-{\langle 1_X, f \rangle} & X \times Y & Y.
\ar[l]^-{\langle 0, 1_Y \rangle} \ar[ul]_b }
\]
Then we have $m_1 a = 1_X$, $m_1 b = 0$ and $m_2 b = 1_Y$. Hence
we get a diagram
\[
\xymatrix@!0@=4em{ X \ar@<.5ex>[r]^-a & M \ar@<.5ex>[l]^-{m_1}
\ar@<-.5ex>[r]_-{m_2} & Y \ar@<-.5ex>[l]_-b }
\]
as in~\eqref{srps}. By assumption, the monomorphism $\langle m_1,
m_2 \rangle$ is also a strong epimorphism, so it is an
isomorphism.
\end{proof}

In general, a given point $(\pi_{1},\Delta_Y)$ can be strong
without being stably strong (Example~\ref{only strong}).
Nevertheless, if all such points are strong (so that $\C$ is
strongly unital), then they are stably strong (by Propositions
1.8.13 and 1.8.14 in~\cite{Borceux-Bourn} and Proposition~\ref{SU
characterisation}). This gives:

\begin{corollary}
If $\C$ is a pointed finitely complete category, then $\SU(\C)=\C$
if and only if $\C$ is strongly unital.\noproof
\end{corollary}

\begin{corollary}
If $\C$ is a pointed finitely complete category and $\SU(\C)$ is
closed under finite limits in $\C$, then $\SU(\C)$ is a strongly
unital category.
\end{corollary}
\begin{proof}
The category $\SU(\C)$ is obviously pointed. Its inclusion into
$\C$ preserves monomorphisms and binary products and it reflects
isomorphisms.
\end{proof}

\begin{proposition}\label{4.8}
If $\C$ is a pointed regular category, then $\SU(\C)$ is closed
under quotients in $\C$.
\end{proposition}
\begin{proof}
This follows readily from Proposition~\ref{stably strong points
closed under quotients}.
\end{proof}

When $\C$ is a regular unital category, an object $Y$ satisfying
condition (iv) of Proposition~\ref{SU characterisation} is called
a
\defn{gregarious} object (Definition~1.9.1 and Theorem~1.9.7 in~\cite{Borceux-Bourn}). So, in
that case, $\SU(\C)$ is precisely the category of gregarious
objects in~$\C$.

\begin{example}\label{Greg}
$\SU(\Mon)=\GMon$, the category of gregarious monoids. A
monoid~$Y$ is gregarious if and only if for all $y\in Y$ there
exist $u$, $v\in Y$ such that $uyv=1$ (Proposition~1.9.2
in~\cite{Borceux-Bourn}). Counterexample~1.9.3
in~\cite{Borceux-Bourn} provides a gregarious monoid which is not
a group: the monoid $Y$ with two generators $x$,~$y$ and the
relation $xy=1$. Indeed $Y=\{y^{n}x^{m}\mid \text{$n$,
$m\in\N$}\}$ and $x^{n}(y^{n}x^{m})y^{m}=1$.
\end{example}

For monoids and the class $\s$ of all stably strong points of
monoids, we have $\s(\Mon)=\SU(\Mon)=\GMon\neq \Gp$ as explained
in Remark~\ref{Stably strong not Schreier}. In particular, there
are in $\Mon$ stably strong points which are not Schreier. Since
$\s(\Mon)$ is not protomodular, it is not a protomodular core with
respect to the class $\s$. Hence for the case of monoids, such a
class $\s$ does not meet our purposes. The major issue here
concerns the closedness of the class $\s$ in $\Pt(\C)$ under
finite limits. To avoid this difficulty, in the next sections our
work focuses more on objects rather than classes.

\section{Unital objects and subtractive objects}\label{USO}
It is known that a pointed finitely complete category is strongly
unital if and only if it is unital and
subtractive~\cite[Proposition 3]{ZJanelidze-Subtractive}. Having
introduced the notion of a strongly unital object, we now explore
analogous notions for the unital and subtractive cases. Our aim is
to prove that the equivalence above also holds ``locally'' for
objects in any pointed regular category.

Let $\C$ be pointed and finitely complete.

\begin{definition} \label{definition U}
Given an object $Y$ of $\C$, we say that $Y$ is \defn{unital} if
the point
\begin{equation*}
(\pi_{1}\colon{Y\times Y\to Y},\quad \langle 1_Y,0 \rangle\colon
{Y\to Y\times Y})
\end{equation*}
is stably strong.

Note that we could equivalently ask that the point
$(\pi_{2},\langle 0, 1_Y \rangle)$ is stably strong. We write
$\U(\C)$ for the full subcategory of $\C$ determined by the unital
objects.
\end{definition}

The following results are proved similarly to the corresponding
ones obtained for strongly unital objects. Recall
from~\cite{Borceux-Bourn,Bourn2002} that a \defn{split punctual
span} is a diagram of the form
\begin{equation}\label{ps}
\xymatrix@!0@=4em{ X \ar@<.5ex>[r]^s & Z \ar@<.5ex>[l]^f
\ar@<-.5ex>[r]_g & Y \ar@<-.5ex>[l]_t }
\end{equation}
where $fs = 1_X$, $gt = 1_Y$, $ft = 0$ and $gs=0$.

\begin{proposition}\label{U characterisation}
If $\C$ is a pointed finitely complete category, then the
following conditions are equivalent:
\begin{tfae}
\item $Y$ is a unital object of $\C$; \item for every object $X$,
the point $(\pi_{X}\colon{X\times Y\to X}, \langle 1_{X},0
\rangle\colon {X\to X\times Y})$ is stably strong; \item for every
object $X$, the point $(\pi_{X}, \langle 1_{X},0 \rangle)$ is
strong; \item given any split punctual span~\eqref{ps}, the map
$\langle f, g \rangle \colon {Z \to X \times Y}$ is a strong
epimorphism.\noproof
\end{tfae}
\end{proposition}

Just as any strongly unital category is always unital, we also
have:

\begin{corollary} \label{strongly unital implies unital}
In a pointed finitely complete category, a strongly unital object
is always unital.
\end{corollary}

\begin{proof}
By Propositions~\ref{SU characterisation} and~\ref{U
characterisation}.
\end{proof}

\begin{corollary}\label{5.4}
If $\C$ is a pointed finitely complete category, then $\U(\C)=\C$
if and only if $\C$ is unital.\noproof
\end{corollary}

\begin{examples}
$\Mon$ and $\SRng$ are not strongly unital, but they are unital,
being J\'onsson--Tarski varieties (see Examples~\ref{Examples
unital}). So, $\U(\Mon)=\Mon$ and $\U(\SRng)=\SRng$.
\end{examples}

\begin{corollary}
If $\C$ is a pointed finitely complete category and $\U(\C)$ is
closed under finite limits in $\C$, then $\U(\C)$ is a unital
category.
\end{corollary}
\begin{proof}
Apply Corollary~\ref{5.4} to $\U(\C)$.
\end{proof}

\begin{proposition}
If $\C$ is a pointed regular category, then $\U(\C)$ is closed
under quotients in $\C$.\noproof
\end{proposition}

\subsection{Subtractive categories, subtractive objects}\label{SC}
We recall the definition of a subtractive category
from~\cite{ZJanelidze-Subtractive}. A relation $r=\langle
r_{1},r_{2}\rangle\colon {R\to X\times Y}$ in a pointed category
is said to be
\defn{left (right) punctual}~\cite{Bourn2002} if $\langle
1_X,0\rangle\colon {X\to X\times Y}$ (respectively $\langle
0,1_Y\rangle\colon {Y\to X\times Y}$) factors through $r$. A
pointed finitely complete category~$\C$ is said to be
\defn{subtractive}, if every left punctual reflexive relation on an object $X$ in
$\C$ is right punctual. It is equivalent to asking that right
punctuality implies left punctuality---which is the implication we
shall use to obtain a definition of subtractivity for objects.

\begin{example}
A variety of universal algebras is subtractive in the sense of
Example~\ref{only strong} if and only if the condition of~\ref{SC}
is satisfied (see \cite{ZJanelidze-Subtractive}).
\end{example}

It is shown in~\cite{ZJanelidze-Snake} that a pointed regular
category $\C$ is subtractive if and only if every span $\langle
s_1,s_2 \rangle \colon A \to B\times C$ is \defn{subtractive}:
written in set-theoretical terms, its induced relation $r=\langle
r_1,r_2 \rangle\colon R\to B\times C$, where $\langle s_1,s_2
\rangle=rp$ for $r$ a monomorphism and $p$ a regular epimorphism,
satisfies the condition
\[
(b,c),\; (b,0)\in R \quad\Rightarrow\quad (0,c)\in R.
\]

\begin{proposition}\label{subtraction via spans}
In a pointed regular category, consider a split right punctual
span~\eqref{srps}. The span $\langle g,f \rangle$ is subtractive
if and only if $f\ker(g)$ is a regular epimorphism.
\end{proposition}
\begin{proof}
Thanks to the Barr embedding theorem \cite{Barr}, in a regular
context it suffices to give a set-theoretical proof (see
Metatheorem~A.5.7 in~\cite{Borceux-Bourn}, for instance). Consider
the factorisation
\[
\xymatrix@!0@C=4em@R=3em{Z\ar[rr]^-{\langle g,f \rangle} \ar@{>>}[dr]_-{p} & & Y\times X \\
 & R \ar@{ >->}[ur]_-{\langle r_1, r_2 \rangle}}
\]
of $\langle g,f\rangle$ as a regular epimorphism $p$ followed by a
monomorphism $\langle r_1, r_2 \rangle$. Then $(y,x)\in R$ if and
only if $y=g(z)$ and $x=f(z)$, for some $z\in Z$.

Suppose that $\langle g, f \rangle$ is subtractive. Given any
$x\in X$, we have $(gs(x),x)\in R$ for $z=s(x)$ and $(gs(x),0)\in
R$ for $z=tgs(x)$. Then $(0,x)\in R$ by assumption, which means
that $0=g(z)$ and $x=f(z)$, for some $z\in Z$. Thus $f\ker(g)$ is
a regular epimorphism.

The converse implication easily follows since $(0,x)\in R$, for
any $x\in X$, because $f\ker(g)$ is a regular epimorphism.
\end{proof}

This result leads us to the following ``local'' definition:

\begin{definition}\label{subtractive object}
Given an object $Y$ of a pointed regular category $\C$, we say
that $Y$ is
\defn{subtractive} when for every split right punctual span~\eqref{srps}, the morphism $f\ker(g)$ is a regular epimorphism.
\end{definition}

We write $\S(\C)$ for the full subcategory of $\C$ determined by
the subtractive objects.

\begin{proposition}
If $\C$ is a pointed regular category, then $\C$ is subtractive if
and only if all of its objects are subtractive.
\end{proposition}
\begin{proof}
As recalled above, if $\C$ is subtractive, then every span is
subtractive. Then every object is subtractive by
Proposition~\ref{subtraction via spans}.

Conversely, consider a right punctual reflexive relation $\langle
r_{1},r_{2}\rangle\colon {R\to X\times X}$. By assumption,
$r_1\ker(r_2)$ is a regular epimorphism. In the commutative
diagram between kernels
\[
\xymatrix@!0@C=6em@R=5em{ K \ar@{ |>->}[r]^-{\ker(r_2)}
\ar@{>>}[d]_-{r_1 \ker(r_2)} \pullback & R \ar@{ >->}[d]^-{\langle
r_1, r_2 \rangle} \ar[r]^-{r_2}
 & X \ar@{=}[d] \\
 X \ar@{ |>->}[r]_-{\langle 1_X, 0 \rangle} & X\times X \ar[r]_-{\pi_2} & X,}
\]
the left square is necessarily a pullback. So, the regular
epimorphism $r_1 \ker(r_2)$ is also a monomorphism, thus an
isomorphism. The morphism $\ker(r_2)$ gives the factorisation of
$\langle 1_{X}, 0\rangle$ needed to prove that $R$ is a left
punctual relation.
\end{proof}

\begin{corollary}
If $\C$ is a pointed regular category and $\S(\C)$ is closed under
finite limits in $\C$, then $\S(\C)$ is a subtractive category.
\end{corollary}
\begin{proof}
Apply the above proposition to $\S(\C)$.
\end{proof}

\begin{proposition}[$\S(\C)\cap \U(\C)=\SU(\C)$]\label{SU=SU}
Let $\C$ be a pointed regular category. An object $Y$ of $\C$ is
strongly unital if and only if it is unital and subtractive.
\end{proposition}
\begin{proof}
We already observed that a strongly unital object is unital
(Corollary~\ref{strongly unital implies unital}). To prove that
$Y$ is subtractive, we consider an arbitrary split right punctual
span such as~\eqref{srps}. In the commutative diagram between
kernels
\[
\xymatrix@!0@C=6em@R=5em{ K \ar@{ |>->}[r]^-{\ker(g)} \ar[d]_-{f
\ker(g)} \pullback & Z \ar[d]^-{\langle f,g \rangle} \ar[r]^-{g}
 & Y \ar@{=}[d] \\
 X \ar@{ |>->}[r]_-{\langle 1_X,0 \rangle} & X\times Y \ar[r]_-{\pi_Y} & Y,}
\]
the left square is necessarily a pullback. By Proposition~\ref{SU
characterisation}, $\langle f,g \rangle$ is a regular epimorphism,
hence so is $f\ker(g)$.

Conversely, given a subtractive unital object $Y$ in a split right
punctual span~\eqref{srps}, by Proposition~\ref{SU
characterisation} we must show that the middle morphism $\langle
f,g \rangle$ of the diagram above is a regular epimorphism. Let
$mp$ be its factorisation as a regular epimorphism $p$ followed by
a monomorphism $m$. The pair $( \langle 1_X,0 \rangle,\langle 0,
1_Y \rangle)$ being jointly strongly epimorphic and $f\ker(g)$
being a regular epimorphism, we see that the pair $(\langle 1_X,0
\rangle f\ker(g),\langle 0,1_Y \rangle)$ is jointly strongly
epimorphic; moreover it factors through the monomorphism $m$.
Consequently, $m$ is an isomorphism.
\end{proof}

\begin{corollary}\label{corollary subtractive}
$\S(\Mon)=\GMon$, $\S(\CMon)=\Ab$ and $\S(\SRng)=\Rng$.
\end{corollary}
\begin{proof}
This is a combination of Examples~\ref{Examples unital} with,
respectively, Example~\ref{Greg}; \cite[Example
1.9.4]{Borceux-Bourn} with Proposition~\ref{SU characterisation}
and the remark following Proposition~\ref{4.8}; and
Theorem~\ref{SU(SRng)=Rng}.
\end{proof}

\begin{example}
Groups are (strongly) unital objects in the category $\Sub$ of
subtraction algebras (Example~\ref{only strong}). In fact, if for
every $y\in Y$ there is a $y^{*}\in Y$ such that $s(0,y^{*})=y$,
then $Y$ is a unital object; in particular, any group is unital.
To see this, we must prove that for any subtraction algebra $X$,
the pair
\[
(\langle1_{X},0\rangle\colon X\to X\times Y,\quad
\langle0,1_{Y}\rangle\colon Y\to X\times Y)
\]
is jointly strongly epimorphic. This follows from the fact that
\[ s((x,0), (0, y^{*})) = (s(x,0), s(0, y^{*})) = (x,y)
\]
for all $x\in X$ and $y\in Y$. Note that the inclusion $\Gp\subset
\SU(\Sub)$ is strict, because the three-element subtraction
algebra
\begin{center}
\begin{tabular}{c|ccccc}
 $s$ & $0$ & $1$ & $2$ \\
\hline
 $0$ & $0$ & $1$ & $2$ \\
 $1$ & $1$ & $0$ & $0$ \\
 $2$ & $2$ & $0$ & $0$
\end{tabular}
\end{center}
satisfies the condition on the existence of $y^{*}$. However, it
is not a group, since the unique group of order three has a
different induced subtraction.
\end{example}

\begin{proposition}
Let $\C$ be a pointed regular category. Then $\S(\C)$ is closed
under quotients in~$\C$.
\end{proposition}
\begin{proof}
Suppose that $Y$ is a subtractive object in $\C$ and consider a
regular epimorphism $w\colon{Y\to W}$. To prove that $W$ is also
subtractive, consider a split right punctual span
$$
\xymatrix@!0@=4em{ X \ar@<.5ex>[r]^s & Z \ar@<.5ex>[l]^f
 \ar@<-.5ex>[r]_g & W; \ar@<-.5ex>[l]_t }
$$
we must prove that $f\ker(g)$ is a regular epimorphism. Consider
the following diagram where all squares are pullbacks:
$$
\xymatrix@!0@C=6em@R=4em{& X' \ar@{.>}[ld]_-{s''} \ar@{>>}[r]^-x \ar@{ >->}[d]_-{s'} \pullback & X \ar@{ >->}[d]^-s \\
 Z'' \pullback \ar@{>>}[r]^-{z'} \ar[d]_-{\langle f'', g'' \rangle} & Z' \pullback \ar@{>>}[r]^-z \ar[d]_-{\langle f', g' \rangle} & Z \ar[d]^-{\langle f, g \rangle} \\
 X'\times Y \ar@{>>}[r]_-{x\times 1_Y} & X\times Y \ar@{>>}[r]_-{1_X\times w} & X\times W.}
$$
Note that from the bottom right pullback we can deduce that the
pullback of $g$ along $w$ is $g'$. Since $f's'=x$, there is an
induced morphism $s''\colon X'\to Z''$ such that $\langle f'', g''
\rangle s''=\langle 1_{X'}, g's' \rangle$ and $z's''=s'$. There is
also an induced morphism $t''\colon Y\to Z''$ such that $\langle
f'', g'' \rangle t''=\langle 0,1_Y \rangle$ and $zz't''=tw$. So,
we get a split right punctual span
$$
\xymatrix@!0@=4em{ X' \ar@<.5ex>[r]^{s''} & Z''
\ar@<.5ex>[l]^{f''}
 \ar@<-.5ex>[r]_{g''} & Y, \ar@<-.5ex>[l]_{t''} }
$$
so that $f''\ker(g'')$ is a regular epimorphism, by assumption.
Since $g'$ is a pullback of $g$ and $g''=g'z'$, we have the
commutative diagram
\[
\xymatrix@!0@C=6em@R=4em{K'' \pullback \ar@{ |>->}[r]^-{\ker(g'')} \ar@{.>}[d]_-{\lambda} & Z'' \ar@{>>}[d]^(.25){z'} \\
 K \pullback \ar@{ |>->}[r]_-{\ker(g')} \ar[d] \ar@(ur,ul)@/^1.75pc/[rr]|(.5)\hole^(0.75){\ker(g)} & Z' \ar[d]^-{g'} \pullback \ar@{>>}[r]^-z & Z \ar[d]^-g \\
 0 \ar[r] & Y \ar@{>>}[r]_-w & W}
\]
between their kernels. Finally, the morphism $xf''\ker(g'')$ is a
regular epimorphism (since both $x$ and $f''\ker(g'')$ are) and
from
\[
xf''\ker(g'') = fzz'\ker(g'')=fz\ker(g')\lambda=f\ker(g)\lambda
\]
we conclude that $f\ker(g)$ is a regular epimorphism, as desired.
\end{proof}

In the presence of binary coproducts, a pointed regular category
$\C$ is subtractive if and only if any split right punctual span
of the form
$$
\xymatrix@!0@=5em{ X \ar@<.5ex>[r]^-{\iota_1} & X+X
\ar@<.5ex>[l]^-{\links 1_X\; 0 \rechts}
 \ar@<-.5ex>[r]_-{\links 1_X\;1_X \rechts} & X \ar@<-.5ex>[l]_-{\iota_2} }
$$
is such that $\delta_X=\links 1_X\; 0 \rechts \ker(\links 1_X\;
1_X \rechts)$ is a regular epimorphism (see Theorem~5.1
in~\cite{DB-ZJ-2009}). This result leads us to the following
characterisation, where an extra morphism $f$ appears as in
Proposition~\ref{SU characterisation}, to be compatible with the
pullback-stability in the definitions of unital and strongly
unital objects.

\begin{proposition} In a pointed regular category $\C$ with binary coproducts the following conditions are equivalent:
\begin{tfae}
 \item an object $Y$ in $\C$ is subtractive;
 \item for any morphism $f\colon X\to Y$, the split right punctual span
 $$ \xymatrix@!0@=5em{ X \ar@<.5ex>[r]^-{\iota_X} & X+Y \ar@<.5ex>[l]^-{\links 1_X\; 0 \rechts}
 \ar@<-.5ex>[r]_-{\links f\; 1_Y \rechts} & Y \ar@<-.5ex>[l]_-{\iota_Y} }
 $$
 is such that $\delta_f=\links 1_X\; 0 \rechts \ker(\links f\;1_Y \rechts)$ is a regular epimorphism.
\end{tfae}
\end{proposition}
\begin{proof}
The implication (i) $\Rightarrow$ (ii) is obvious. Conversely,
given any split right punctual span~\eqref{srps}, we have a
morphism $gs\colon X\to Y$, so for the split right punctual span
$$ \xymatrix@!0@=5em{ X \ar@<.5ex>[r]^-{\iota_X} & X+Y \ar@<.5ex>[l]^-{\links 1_X\; 0 \rechts}
 \ar@<-.5ex>[r]_-{\links gs\; 1_Y \rechts} & Y \ar@<-.5ex>[l]_-{\iota_Y} }
$$
we have that $\delta_{gs}=\links 1_X\;0 \rechts \ker(\links
gs\;1_Y \rechts)$ is a regular epimorphism. The induced morphism
$\sigma$ between kernels in the diagram
\[
\xymatrix@!0@C=7em@R=5em{ K \ar@{ |>->}[r]^-{\ker(\links gs\;1_Y \rechts)} \ar@{.>}[d]_-{\sigma} \pullback & X+Y \ar[d]^-{\links s\; t\rechts} \ar[r]^-{\links gs\;1_Y \rechts} & Y \ar@{=}[d] \\
 K_g \ar@{ |>->}[r]_-{\ker(g)} & Z \ar[r]_-{g} & Y}
\]
is such that $f\ker(g)\sigma = f\links s\; t \rechts \ker(\links
gs\;1_Y \rechts)=\delta_{gs}$ is a regular epimorphism;
consequently, $f\ker(g)$ is a regular epimorphism as well.
\end{proof}



\section{Mal'tsev objects}\label{MO}
Even though the concept of a strongly unital object is strong
enough to characterise rings amongst semirings as in
Theorem~\ref{SU(SRng)=Rng}, it fails to give us a characterisation
of groups amongst monoids. For that purpose we need a stronger
concept. The aim of the present section is two-fold: first to
introduce Mal'tsev objects, then to prove that any Mal'tsev object
in the category of monoids is a group (Theorem~\ref{Mal'tsev
monoids are groups}). In fact, also the opposite inclusion holds:
groups are precisely the Mal'tsev monoids. This follows from the
results in the next section, where the even stronger concept of a
protomodular object is introduced.

\begin{definition} \label{definition Mal'tsev objects}
We say that an object $Y$ of a finitely complete category $\C$ is
a
\defn{Mal'tsev object} if the category $\Pt_Y(\C)$ is unital.

As explained after Proposition~\ref{Mal'tsev via fibres}, this
means that for every pullback of split epimorphisms over $Y$ as
in~\eqref{pb of split epis}, the morphisms $\langle 1_{A}, tf
\rangle$ and $\langle sg, 1_{C} \rangle$ are jointly strongly
epimorphic.
\end{definition}

We write $\M(\C)$ for the full subcategory of $\C$ determined by
the Mal'tsev objects.

\begin{proposition}\label{double split epi Mal'tsev}
Let $\C$ be a regular category. For any object $Y$ in $\C$, the
following conditions are equivalent:
\begin{tfae}
\item $Y$ is a Mal'tsev object; \item every double split
epimorphism
\[
\xymatrix@!0@=4em{ D \ar@<-.5ex>[d]_{g'} \ar@<-.5ex>[r]_{f'} & C
\ar@<-.5ex>[d]_g
\ar@<-.5ex>[l]_-{s'} \\
A \ar@<-.5ex>[u]_{t'} \ar@<-.5ex>[r]_f & Y \ar@<-.5ex>[l]_s
\ar@<-.5ex>[u]_t }
\]
over $Y$ is a regular pushout; \item every double split
epimorphism over $Y$ as above, in which $f'$ and $g'$ are jointly
monomorphic, is a pullback.
\end{tfae}
\end{proposition}
\begin{proof}
The equivalence between (ii) and (iii) is immediate.

(i) $\Rightarrow$ (ii). Consider a double split epimorphism over
$Y$ as above. We want to prove that the comparison morphism
$\langle g', f' \rangle \colon D \to A\times_Y C$ is a regular
epimorphism. Suppose that $\langle g',f' \rangle =me$ is its
factorisation as a regular epimorphism followed by a monomorphism.
We obtain the commutative diagram
\[
\xymatrix@!0@C=5em@R=4em{ & M \ar@{ >->}[d]^- m \\
 A \ar[r]_-{\langle 1_A, tf\rangle} \ar[ur]^-{et'} & A\times_Y C & C. \ar[l]^-{\langle sg,1_C\rangle} \ar[ul]_-{es'}}
 \]
By assumption $(\langle 1_A, tf\rangle, \langle sg,1_C\rangle)$ is
jointly strongly epimorphic, which proves that~$m$ is an
isomorphism and, consequently, $\langle g',f' \rangle$ is a
regular epimorphism.

(ii) $\Rightarrow$ (i). Consider a pullback of split epimorphisms
\eqref{pb of split epis} and a monomorphism~$m$ such that $\langle
1_A,tf \rangle$ and $\langle sg,1_C \rangle$ factor through $m$
$$
 \xymatrix@!0@C=5em@R=4em{ & M \ar@{ >->}[d]^- m \\
 A \ar[r]_-{\langle 1_A, tf\rangle} \ar[ur]^-{a} & A\times_Y C & C. \ar[l]^-{\langle sg,1_C\rangle} \ar[ul]_-{c}}
$$
We obtain a double split epimorphism over $Y$ given by
\[
\xymatrix@!0@=5em{ M \ar@<-.5ex>[d]_{\pi_A m}
\ar@<-.5ex>[r]_{\pi_C m} & C \ar@<-.5ex>[d]_g
\ar@<-.5ex>[l]_-{c} \\
A \ar@<-.5ex>[u]_{a} \ar@<-.5ex>[r]_f & Y, \ar@<-.5ex>[l]_s
\ar@<-.5ex>[u]_t }
\]
whose comparison morphism to the pullback of $f$ and $g$ is
$m\colon {M \to A\times_Y C}$. By assumption, $m$ is a regular
epimorphism, hence it is an isomorphism.
\end{proof}

\begin{proposition} \label{Mal'tsev implies strongly unital}
Let $\C$ be a pointed regular category. Every Mal'tsev object in
$\C$ is a strongly unital object.
\end{proposition}

\begin{proof}
Let $Y$ be a Mal'tsev object. By Proposition~\ref{SU
characterisation}, given a split right punctual span
\[
\xymatrix@!0@=4em{ X \ar@<.5ex>[r]^s & Z \ar@<.5ex>[l]^f
\ar@<-.5ex>[r]_g & Y \ar@<-.5ex>[l]_t }
\]
we need to prove that the the morphism $\langle f, g \rangle
\colon {Z \to X \times Y}$ is a strong epimorphism. Consider the
commutative diagram on the right
\begin{equation*}
\vcenter{\xymatrix@!0@=4em{\Eq(f) \ar@<-1ex>[r]
\ar@<-.5ex>[d]_-{g'} \ar@<1ex>[r]^-{\pi_{1}} & Z \ar[l]
\ar@<-.5ex>[d]_-{g} \ar@{->>}[r]^-{f} & X
\ar@<-.5ex>[d] \\
\Eq(!_{Y}) \ar@<-.5ex>[u]_-{t'} \ar@<-1ex>[r]
\ar@<1ex>[r]^-{\pi_{1}}
 & Y \ar[l] \ar@<-.5ex>[u]_-{t} \ar@{->>}[r]_-{!_{Y}} & 0 \ar@<-.5ex>[u]}}
\end{equation*}
and take kernel pairs to the left. Note that the square on the
right is a regular epimorphism of points. Since $Y$ is a Mal'tsev
object, by Proposition~\ref{double split epi Mal'tsev} the double
split epimorphism of first (or second) projections on the left is
a regular pushout. Lemma~\ref{Bourn Lemma} tells us that the
square on the right is a regular pushout as well, which means that
the morphism $\langle f, g \rangle \colon {Z \to X \times Y}$ is a
regular, hence a strong, epimorphism.
\end{proof}

For a pointed finitely complete category $\C$, the category
$\SU(\C)$ obviously contains the zero object. By the following
proposition we see that the zero object is not necessarily a
Mal'tsev object. Hence if $\C$ is pointed and regular, but not
unital, then $\M(\C)$ is strictly contained in $\SU(\C)$.

\begin{proposition}
If $\C$ is a pointed finitely complete category, then the zero
object is a Mal'tsev object if and only if $\C$ is unital.
\end{proposition}
\begin{proof}
The zero object $0$ is a Mal'tsev object if and only if, for any
$X$, $Y \in \C$, in the diagram
\[
\xymatrix@!0@=5em{X \times Y \splitsplitpullback
\ar@<-.5ex>[d]_{\pi_X} \ar@<-.5ex>[r]_(.7){\pi_Y} & Y
\ar@<-.5ex>[l]_-{\langle 0, 1_{Y} \rangle}
\ar@<-.5ex>[d] \\
X \ar@<-.5ex>[u]_(.4){\langle 1_{X}, 0 \rangle} \ar@<-.5ex>[r] &
0, \ar@<-.5ex>[l] \ar@<-.5ex>[u] }
\]
the morphisms $\langle 1_{X}, 0 \rangle$ and $\langle 0, 1_{Y}
\rangle$ are jointly strongly epimorphic. This happens if and only
if $\C$ is unital.
\end{proof}

\begin{remark}
By Proposition~\ref{Mal'tsev via fibres}, $\C$ is a Mal'tsev
category if and only if all fibres $\Pt_Y(\C)$ are unital if and
only if they are strongly unital. For a Mal'tsev object $Y$ in a
category $\C$ the fibre $\Pt_Y(\C)$ is unital, but not strongly
unital in general. The previous proposition provides a
counterexample: if $\C=\Mon$ and $Y=0$, then~$Y$ is a Mal'tsev
object, but the category $\Pt_Y(\Mon)=\Mon$ is not strongly
unital~\cite[Example 1.8.2]{Borceux-Bourn}.
\end{remark}

Next we see that some well-known properties which hold for
Mal'tsev categories are still true for Mal'tsev objects.
\begin{proposition}
In a finitely complete category, a reflexive graph whose object of
objects is a Mal'tsev object admits at most one structure of
internal category.
\end{proposition}
\begin{proof}
Given a reflexive graph
\[
\xymatrix@!0@=4em{ X_1 \ar@<1ex>[r]^d \ar@<-1ex>[r]_c & X \ar[l]|e
}
\]
where $X$ is a Mal'tsev object, let $m \colon {X_2 \to X_1}$ be a
multiplication, where $X_2$ is the object of composable arrows. If
this multiplication endows the graph with a structure of internal
category, then it must be compatible with the identities, which
means that
\begin{equation} \label{internal category}
m \langle 1_{X_1}, ec \rangle = m \langle ed, 1_{X_1} \rangle =
1_{X_1}.
\end{equation}
Considering the pullback
\[
\xymatrix@!0@=5em{ X_2 \splitsplitpullback \ar@<-.5ex>[d]_{\pi_1}
\ar@<-.5ex>[r]_{\pi_2} & X_1 \ar@<-.5ex>[d]_d
\ar@<-.5ex>[l]_-{\langle ed ,1_{X_1} \rangle} \\
X_1 \ar@<-.5ex>[u]_(.4){\langle 1_{X_1}, ec \rangle}
\ar@<-.5ex>[r]_c & X, \ar@<-.5ex>[l]_e \ar@<-.5ex>[u]_e }
\]
we see that $\langle 1_{X_1}, ec \rangle$ and $\langle ed, 1_{X_1}
\rangle$ are jointly (strongly) epimorphic, because~$X$ is a
Mal'tsev object. Then there is at most one morphism $m$ satisfying
the equalities~\eqref{internal category}.
\end{proof}

\begin{proposition}
In a finitely complete category, any reflexive relation on a
Mal'\-tsev object is transitive.
\end{proposition}
\begin{proof}
The proof is essentially the same as that of~\cite[Proposition
5.3]{S-proto}.
\end{proof}

\begin{example}
Unlike the case of Mal'tsev categories, it is not true that every
internal category with a Mal'tsev object of objects is a groupoid.
Neither is it true that every reflexive relation on a Mal'tsev
object is symmetric. The category $\Mon$ of monoids provides
counterexamples. Indeed, as we show below in
Theorems~\ref{Mal'tsev monoids are groups} and~\ref{groups =
protomodular monoids}, the Mal'tsev objects in $\Mon$ are
precisely the groups. As a consequence of Propositions 2.2.4 and
3.3.2 in~\cite{SchreierBook}, in $\Mon$ an internal category over
a group is a groupoid if and only if the kernel of the domain
morphism is a group. Similarly, a reflexive relation on a group is
symmetric if and only if the kernels of the two projections of the
relation are groups. A~concrete example of a (totally
disconnected) internal category which is not a groupoid is the
following. If $M$ is a commutative monoid and $G$ is a group,
consider the reflexive graph
\[
\xymatrix@!0@=7em{ M \times G \ar@<1ex>[r]^-{\pi_G}
\ar@<-1ex>[r]_-{\pi_G} & G. \ar[l]|-{\langle 0, 1_{G} \rangle} }
\]
It is an internal category by Proposition 3.2.3 in
\cite{SchreierBook}, but in general it is not a groupoid, since
the kernel of $\pi_G$, which is $M$, need not be a group.
\end{example}

\begin{proposition}\label{RS=SR}
In a regular category, any pair of reflexive relations $R$ and $S$
on a Mal'tsev object $Y$ commutes: $RS=SR$.
\end{proposition}
\begin{proof}
The proof of this result is similar to that of Proposition 2.8
in~\cite{Bourn2014}. Consider the double relation $R\square S$ on
$R$ and $S$:
$$
\xymatrix@!0@=7em{R\square S \ar@<-1ex>[d]_{\pi_{12}}
\ar@<1ex>[d]^-{\pi_{34}} \ar@<-1ex>[r]_-{\pi_{24}}
\ar@<1ex>[r]^-{\pi_{13}} &
 S \ar@<-1ex>[d]_{s_1} \ar@<1ex>[d]^-{s_2} \ar[l] \\
 R \ar@<-1ex>[r]_-{r_1} \ar@<1ex>[r]^-{r_2} \ar[u] & Y. \ar[l] \ar[u]}
$$
In set-theoretical terms, $R\square S$ is given by the subobject
of $Y\times Y\times Y\times Y$ whose elements are quadruples
$(a,b,c,d)$ such that
$$
\begin{array}{ccc} a & \!\!S\!\! & c \\ R & & R \\ b & \!\!S\!\! & d. \end{array}
$$
Let $R\times_Y S$ denote the pullback of $r_2$ and $s_1$, and
$S\times_Y R$ the pullback of $s_2$ and~$r_1$. By
Proposition~\ref{double split epi Mal'tsev}, the comparison
morphisms $\langle \pi_{12},\pi_{24} \rangle \colon R\square S \to
R\times_Y S$ and $\langle \pi_{13},\pi_{34} \rangle \colon
R\square S \to S\times_Y R$ are regular epimorphisms. Applying
Proposition~2.3 in~\cite{Bourn-Gran-Normal-Sections} to these
regular epimorphisms, it easily follows that $SR\leq RS$
and~${RS\leq SR}$.
\end{proof}

\begin{proposition}\label{Mal'tsev objects Mal'tsev cat}
If $\C$ is a finitely complete category, then $\M(\C)=\C$ if and
only if $\C$ is a Mal'tsev category.
\end{proposition}
\begin{proof}
By Proposition~\ref{Mal'tsev via fibres}.
\end{proof}

\begin{corollary} \label{Mal'tsev objects Mal'tsev cat 2}
If $\C$ is a finitely complete category and $\M(\C)$ is closed
under finite limits in $\C$, then $\M(\C)$ is a Mal'tsev category.
\end{corollary}
\begin{proof}
Apply Proposition~\ref{Mal'tsev objects Mal'tsev cat} to $\M(\C)$.
\end{proof}

\begin{proposition}\label{Mal'tsev closed for quots}
If $\C$ is a regular category, then $\M(\C)$ is closed under
quotients in $\C$.
\end{proposition}
\begin{proof}
Given a Mal'tsev object $X$ and a regular epimorphism $f\colon
{X\to Y}$, any double split epimorphism over $Y$ may be pulled
back to a double split epimorphism over $X$, which is a regular
pushout by assumption. It is straightforward to check that the
given double split epimorphism over $Y$ is then a regular pushout.
\end{proof}

\begin{example}
As a consequence of Example~\ref{semirings PM} below, in the
category of semi\-rings the Mal'tsev objects are precisely the
rings: $\M(\SRng)=\SU(\SRng)=\Rng$.
\end{example}

\begin{theorem} \label{Mal'tsev monoids are groups}
If $\C$ is the category $\Mon$ of monoids, then $\M(\C)$ is
contained in the subcategory $\Gp$ of groups. In other words, if
the category $\Pt_{M}(\Mon)$ is unital then the monoid $M$ is a
group.
\end{theorem}

\begin{proof}
Let $M$ be a Mal'tsev object in the category of monoids. Given any
element $m \neq e_{M}$ of $M$, we are going to prove that it is
right invertible. This suffices for the monoid $M$ to be a group.

Consider the pullback diagram
\begin{equation}\label{M(Mon)}
\vcenter{\xymatrix@!0@=5em{ P \splitsplitpullback
\ar@<-.5ex>[r]_(.5){\pi_2} \ar@<-.5ex>[d]_-{\pi_1} & M + M
\ar@<-.5ex>[l]_-{i_2}
\ar@<-.5ex>[d]_-{\links1_M\;1_M\rechts} \\
M + \N \ar@<-.5ex>[u]_-{i_1} \ar@<-.5ex>[r]_-{\links1_M\;m\rechts}
& M \ar@<-.5ex>[l]_-{\iota_M} \ar@<-.5ex>[u]_-{\iota_1} }}
\end{equation}
where $m \colon \N \to M$ is the morphism sending $1$ to $m$.

Recall that $M + M$ may be seen as the set of words of the form
\[
\underline{l}_1\sqbullet \overline{r}_1\sqbullet \cdots \sqbullet
\underline{l}_s \sqbullet \overline{r}_s
\]
for $\underline{l}_i$, $\overline{r}_i \in M$, subject to the rule
that we may multiply underlined with underlined elements and
overlined with overlined ones, or any of such with the neutral
element $e_M$. The two coproduct inclusions can be described as
\[ \iota_1(l) = \underline{l} \qquad \iota_2(r) = \overline{r}
\]
for $l$, $r\in M$. We use essentially the same notations for the
elements of $M+\N$, writing a generic element as
$\underline{m}_1\sqbullet \overline{n}_1\sqbullet \cdots \sqbullet
\underline{m}_t\sqbullet \overline{n}_t$.

We see that the pullback $P$ consists of pairs
\[
(\underline{m}_1\sqbullet \overline{n}_1\sqbullet \cdots \sqbullet
\underline{m}_t\sqbullet \overline{n}_t,\;\underline{l}_1\sqbullet
\overline{r}_1\sqbullet \cdots \sqbullet \underline{l}_s \sqbullet
\overline{r}_s)\in (M+\N)\times(M+M)
\]
such that $m_1 m^{n_1} \cdots m_t m^{n_t}=l_1r_1 \cdots l_s r_s$.
We also know that
\begin{align*}
i_1(\underline{m}_1\sqbullet \overline{n}_1\sqbullet \cdots
\sqbullet \underline{m}_t\sqbullet \overline{n}_t) &=
(\underline{m}_1\sqbullet \overline{n}_1\sqbullet \cdots \sqbullet
\underline{m}_t\sqbullet \overline{n}_t,\;
\underline{m_1 m^{n_1} \cdots m_t m^{n_t}}),\\
i_2(\underline{l}_1\sqbullet \overline{r}_1\sqbullet \cdots
\sqbullet \underline{l}_s \sqbullet \overline{r}_s) &=
(\underline{l_1r_1 \cdots l_s r_s},\; \underline{l}_1\sqbullet
\overline{r}_1\sqbullet \cdots \sqbullet \underline{l}_s \sqbullet
\overline{r}_s).
\end{align*}
Note that $(\overline{1}, \overline{m})$ belongs to $P$, where
$\overline{1}$ is our way to view $1 \in \N$ as an element
of~$M+\N$. Since by assumption $i_1$ and $i_2$ are jointly
strongly epimorphic, we have
\begin{align*} \label{equation Mal'tsev monoid}
(\overline{1}, \overline{m})=\,&(\underline{m}_1^1\sqbullet
\overline{n}_1^1\sqbullet \cdots \sqbullet
\underline{m}_{t_1}^1\sqbullet \overline{n}_{t_1}^1,
\underline{m_1^1
m^{n_1^1} \cdots m_{t_1}^1 m^{n_{t_1}^1}})\\
 &\sqbullet(\underline{l^{1}_1r^{1}_1 \cdots l^{1}_{s_{1}} r^{1}_{s_{1}}},\;
\underline{l}^{1}_1\sqbullet \overline{r}^{1}_1\sqbullet \cdots
\sqbullet
\underline{l}^{1}_{s_{1}} \sqbullet \overline{r}^{1}_{s_{1}})\\
&\;\vdots \\
&\sqbullet(\underline{m}_1^k\sqbullet \overline{n}_1^k\sqbullet
\cdots \sqbullet \underline{m}_{t_k}^k\sqbullet
\overline{n}_{t_k}^k, \underline{m_1^k m^{n_1^k}
\cdots m_{t_k}^k m^{n_{t_k}^k}})\\
&\sqbullet(\underline{l^{k}_1r^{k}_1 \cdots l^{k}_{s_{k}}
r^{k}_{s_{k}}},\; \underline{l}^{k}_1\sqbullet
\overline{r}^{k}_1\sqbullet \cdots \sqbullet
\underline{l}^{k}_{s_{k}} \sqbullet \overline{r}^{k}_{s_{k}})
\end{align*}
for some $m^{i}_{j}$, $l^{i}_{j}$, $r^{i}_{j}\in M$ and
$n^{i}_{j}\in \N$. Computing the first component we get that
$\overline{1}$ is equal to
\[
\underline{m}_1^1\sqbullet \overline{n}_1^1\sqbullet \cdots
\sqbullet \underline{m}_{t_1}^1\sqbullet \overline{n}_{t_1}^1
\sqbullet \underline{l^{1}_1r^{1}_1 \cdots l^{1}_{s_{1}}
r^{1}_{s_{1}}} \sqbullet \cdots \sqbullet
\underline{m}_1^k\sqbullet \overline{n}_1^k\sqbullet \cdots
\sqbullet \underline{m}_{t_k}^k\sqbullet \overline{n}_{t_k}^k
\sqbullet \underline{l^{k}_1r^{k}_1 \cdots l^{k}_{s_{k}}
r^{k}_{s_{k}}}.
\]
Since $1$ cannot be written as a sum
$n^{1}_{1}+\cdots+n^{k}_{t_{k}}$ in $\N$ unless all but one of the
$n^{i}_{j}$ is zero, we see that the equality above reduces to
$(\overline{1}, \overline{m})$ being equal to
\begin{multline*}
(\underline{l_1r_1 \cdots l_s r_s},\; \underline{l}_1\sqbullet
\overline{r}_1\sqbullet \cdots \sqbullet \underline{l}_s \sqbullet
\overline{r}_s) \sqbullet(\overline{1}, \underline{m})\sqbullet
(\underline{l'_1r'_1 \cdots l'_{s'} r'_{s'}},\;
\underline{l}'_1\sqbullet \overline{r}'_1\sqbullet \cdots
\sqbullet \underline{l}'_{s'} \sqbullet \overline{r}'_{s'}).
\end{multline*}
Equality of the first components gives us
\[
\overline{1}=\underline{l_1r_1 \cdots l_s r_s} \sqbullet
\overline{1}\sqbullet \underline{l'_1r'_1 \cdots l'_{s'} r'_{s'}}
\]
from which we deduce that
\begin{equation}\label{in K}
l_1r_1 \cdots l_s r_s=e_{M}=l'_1r'_1 \cdots l'_{s'} r'_{s'}.
\end{equation}
This means that $\underline{l}_1\sqbullet \overline{r}_1\sqbullet
\cdots \sqbullet \underline{l}_s \sqbullet \overline{r}_s$ and
$\underline{l}'_1\sqbullet \overline{r}'_1\sqbullet \cdots
\sqbullet \underline{l}'_{s'} \sqbullet \overline{r}'_{s'}$ are in
the kernel of $\links1_M\;1_M\rechts\colon {M+M\to M}$. Without
loss of generality we may assume that these two products are
written in their reduced form, meaning that no further
simplification is possible, besides perhaps when
$\underline{l}_{1}$, $\overline{r}_{s}$, $\underline{l}'_{1}$ or
$\overline{r}'_{s'}$ happens to be equal to $e_{M}$. Computing the
second component, we see that
\begin{align*}
\overline{m} &= \underline{l}_1\sqbullet \overline{r}_1\sqbullet
\cdots \sqbullet \underline{l}_s \sqbullet \overline{r}_s
\sqbullet \underline{m} \sqbullet \underline{l}'_1\sqbullet
\overline{r}'_1\sqbullet \cdots \sqbullet
\underline{l}'_{s'} \sqbullet \overline{r}'_{s'} \\
&= \underline{l}_1\sqbullet \overline{r}_1\sqbullet \cdots
\sqbullet \underline{l}_s \sqbullet \overline{r}_s \sqbullet
\underline{ml'_1}\sqbullet \overline{r}'_1\sqbullet \cdots
\sqbullet \underline{l}'_{s'} \sqbullet \overline{r}'_{s'}.
\end{align*}
This leads to a proof that $m$ is right invertible. Indeed, for
such an equality to hold, certain cancellations must be possible
so that the overlined elements can get together on the right. Next
we study four basic cases which all others reduce to.

\emph{Case $s=s'=1$.} For the equality
\[
\overline{m} = \underline{l}_1\sqbullet \overline{r}_1 \sqbullet
\underline{ml'_1}\sqbullet \overline{r}'_1
\]
to hold, we must have $\overline{r}_1=e_M$ or
$\underline{ml'_1}=e_M$. In the latter situation, $m$ is right
invertible. If, on the other hand, $\overline{r}_1=e_M$, then
$\underline{l}_1=e_M$ by~\eqref{in K}. The equality $\overline{m}
= \underline{ml'_1}\sqbullet \overline{r}'_1$ implies that
$\underline{ml'_1}=e_{M}$.

\emph{Case $s=2$, $s'=1$.} For the equality
\[
\overline{m} = \underline{l}_1\sqbullet \overline{r}_1 \sqbullet
\underline{l}_2\sqbullet \overline{r}_2 \sqbullet
\underline{ml'_1}\sqbullet \overline{r}'_1
\]
to hold, we must have one of the ``inner'' elements on the right
side of the equality equal to $e_M$.
\begin{itemize}
 \item If $\underline{ml'_1}=e_M$, then $m$ is right invertible.
 \item If $\overline{r}_1=e_M$ or $\underline{l}_2=e_M$, then the word $\underline{l}_1\sqbullet \overline{r}_1\sqbullet
\underline{l}_2\sqbullet \overline{r}_2$ is not reduced.
 \item If $\overline{r}_2=e_M$, then $\overline{m} = \underline{l}_1\sqbullet
\overline{r}_1 \sqbullet \underline{l_2ml'_1}\sqbullet
\overline{r}'_1$. Since $\overline{r}_1$ is different from $e_M$,
we have that $l_2ml'_1=e_M$, so that $l_2$ admits an inverse on
the right and $l'_1$ admits one on the left. From~\eqref{in K}, we
also know that $l_2$ is admits an inverse on the left and $l'_1$
admits one on the right. Thus, they are both invertible elements,
and hence so is $m$.
\end{itemize}

\emph{Case $s=1$, $s'=2$.} For the equality
\[
\overline{m} = \underline{l}_1\sqbullet \overline{r}_1\sqbullet
\underline{ml'_1}\sqbullet \overline{r}'_1 \sqbullet
\underline{l}'_2\sqbullet \overline{r}'_2
\]
to hold, we must have one of the ``inner'' elements on the right
side of the equality equal to $e_M$.
\begin{itemize}
 \item If $\underline{ml'_1}=e_M$, then $m$ is right invertible.
 \item If $\overline{r}'_1=e_M$ or $\underline{l}'_2=e_M$, then the word $\underline{l'_1}\sqbullet \overline{r}'_1
\sqbullet \underline{l}'_2\sqbullet \overline{r}'_2$ is not
reduced.
 \item If $\overline{r}_1=e_M$, then $\underline{l}_1=e_M$ by~\eqref{in K}, so that
$\overline{m} = \underline{ml'_1}\sqbullet \overline{r}'_1
\sqbullet \underline{l}'_2\sqbullet \overline{r}'_2$. This is
impossible, since $\overline{r}'_1$ and $\underline{l}'_2$ are
non-trivial.
\end{itemize}

\emph{Case $s=2$, $s'=2$.} For the equality
\[
\overline{m} = \underline{l}_1\sqbullet \overline{r}_1\sqbullet
\underline{l}_2\sqbullet
\overline{r}_2\sqbullet\underline{ml'_1}\sqbullet \overline{r}'_1
\sqbullet \underline{l}'_2\sqbullet \overline{r}'_2
\]
to hold, we must have one of the ``inner'' elements on the right
side of the equality equal to $e_M$.
\begin{itemize}
 \item If $\underline{ml'_1}=e_M$, then $m$ is right invertible.
 \item If $\overline{r}_1=e_M$ or $\underline{l}_2=e_M$, then the word
 $\underline{l}_1\sqbullet \overline{r}_1\sqbullet
\underline{l}_2\sqbullet \overline{r}_2$ is not reduced.
 \item If $\overline{r}'_1=e_M$ or $\underline{l}'_2=e_M$, then the word $\underline{l'_1}\sqbullet \overline{r}'_1
\sqbullet \underline{l}'_2\sqbullet \overline{r}'_2$ is not
reduced.
 \item If $\overline{r}_2=e_M$, then $\overline{m} = \underline{l}_1\sqbullet
\overline{r}_1\sqbullet \underline{l_2ml'_1}\sqbullet
\overline{r}'_1 \sqbullet \underline{l}'_2\sqbullet
\overline{r}'_2$. Again, $\underline{l_2ml'_1}=e_{M}$ as in the
second case, and~\eqref{in K} implies that $m$ is invertible.
\end{itemize}

We see that the last case reduces to one of the previous ones and
it is straightforward to check that the same happens for general
$s$, $s'\geq 2$.
\end{proof}

Below, in Theorem~\ref{groups = protomodular monoids}, we shall
prove that groups are precisely the Mal'tsev monoids: $\M(\Mon) =
\Gp$.

\subsection{$\M(\C)$ is a Mal'tsev core}
As we already recalled in Section~\ref{section S-Mal'tsev and
S-protomodular}, if $\C$ is an $\s$-Mal'tsev category, then the
subcategory of $\s$-special objects $\s(\C)$ is a Mal'tsev
category, called the Mal'tsev core of $\C$ relatively to~$\s$. We
now show that the subcategory $\M(\C)$ of Mal'tsev objects is a
Mal'tsev core with respect to a suitable class $\m$ of points,
provided that $\M(\C)$ is closed under finite limits in $\C$.

Let $\C$ be a finitely complete category such that $\M(\C)$ is
closed under finite limits. We define $\m$ as the class of points
$(f,s)$ in $\C$ for which there exists a pullback of split
epimorphisms
\begin{equation}
\label{diagram for MM} \vcenter{
\xymatrix@!0@=4em{ A \splitsplitpullback \ar@<-.5ex>[r] \ar@<-.5ex>[d]_-{f} & A' \ar@<-.5ex>[l]_-{a} \ar@<-.5ex>[d]_{f'} \\
 X \ar@<-.5ex>[r] \ar@<-.5ex>[u]_-{s} & X', \ar@<-.5ex>[l] \ar@<-.5ex>[u]_{s'}
}}
\end{equation}
for some point $(f',s')$ in $\M(\C)$. Note that the class $\m$ is
obviously stable under pullbacks along split epimorphisms.
Moreover, all points in $\M(\C)$ belong to $\m$.

\begin{proposition}
\label{MM-Mal'tsev} Let $\C$ be a finitely complete category.
Given any pullback of split epimorphisms with $(f,s)$ a point in
$\m$
\[
\xymatrix@!0@=5em{ A\times_{X}C \splitsplitpullback
\ar@<-.5ex>[d]_{\pi_A} \ar@<-.5ex>[r]_(.7){\pi_C} & C
\ar@<-.5ex>[d]_-g
\ar@<-.5ex>[l]_-{\langle sg,1_C \rangle} \\
A \ar@<-.5ex>[u]_(.4){\langle 1_A,tf \rangle} \ar@<-.5ex>[r]_-f &
X, \ar@<-.5ex>[l]_-s \ar@<-.5ex>[u]_-t }
\]
the pair $(\langle 1_A,tf\rangle, \langle sg, 1_C\rangle)$ is
jointly strongly epimorphic.
\end{proposition}
\begin{proof}
Since $(f,s)$ is a pullback of a point in $\M(\C)$ as in
\eqref{diagram for MM}, we see that the pair $(\langle
1_A,tf\rangle a, \langle sg, 1_C\rangle)$ is jointly strongly
epimorphic. It easily follows that also $(\langle 1_A,tf\rangle,
\langle sg, 1_C\rangle)$ is jointly strongly epimorphic.
\end{proof}

Note that the property above already occurred in
Definition~\ref{S-Mal'tsev and S-protomodular categories}(1).

\begin{proposition} \label{Mal'tsev objs = Mal'tsev core}
If $\C$ is a pointed finitely complete category, and the
subcategory $\M(\C)$ of Mal'tsev objects is closed under finite
limits in~$\C$, then it coincides with the subcategory $\m(\C)$ of
$\m$-special objects of~$\C$.
\end{proposition}
\begin{proof}
If $X$ is a Mal'tsev object, it is obviously $\m$-special, since
the point
\begin{equation*}
(\pi_{2}\colon{X\times X\to X},\quad \Delta_{X}=\langle
1_{X},1_{X} \rangle\colon {X\to X\times X})
\end{equation*}
belongs to the subcategory $\M(\C)$, which is closed under binary
products.

Conversely, suppose that $X$ is $\m$-special. Then there is a
point $(f',s')$ in $\M(\C)$ and a point $X\leftrightarrows B'$ in
$\C$ such that the square
\[
\xymatrix@!0@=5em{ X\times X \splitsplitpullback \ar@<-.5ex>[r] \ar@<-.5ex>[d]_{\pi_1} & A' \ar@<-.5ex>[l] \ar@<-.5ex>[d]_{f'} \\
 X \ar@<-.5ex>[u]_(.4){\langle 1_{X},1_{X} \rangle} \ar@<-.5ex>[r] & B' \ar@<-.5ex>[l] \ar@<-.5ex>[u]_{s'} }
\]
is a pullback. But then $X$, which is the kernel of~$\pi_1$, is
also the kernel of $f'$, and hence it belongs to $\M(\C)$.
\end{proof}

Strictly speaking, we cannot apply Proposition 4.3 in
\cite{Bourn2014} to conclude that $\M(\C)$ is the Mal'tsev core of
$\C$ relatively to $\m$, since the class $\m$ we are considering
does not satisfy all the conditions of Definition~\ref{S-Mal'tsev
and S-protomodular categories}. Indeed, our class $\m$ is not
stable under pullbacks, neither need it to be closed in $\Pt(\C)$
under finite limits, in general. However, all the arguments of the
proof of Proposition~4.3 in~\cite{Bourn2014} are still applicable
to our context, since, by definition of the class~$\m$, we know
that every point between objects in $\M(\C)$ belongs to~$\m$. So,
we can conclude that, if $\M(\C)$ is closed in $\C$ under finite
limits, then it is a Mal'tsev category, being the Mal'tsev core of
$\C$ relatively to the class $\m$. Observe that we could also
conclude that $\M(\C)$ is a Mal'tsev category simply by
Corollary~\ref{Mal'tsev objects Mal'tsev cat 2}.

\section{Protomodular objects}\label{Protomodular objects}
In this final section we introduce the (stronger) concept of a
protomodular object and prove our paper's main result,
Theorem~\ref{groups = protomodular monoids}: a monoid is a group
if and only if it is a protomodular object, and if and only if it
is a Mal'tsev object.

\begin{definition} \label{definition protomodular objects}
Given an object $Y$ of a finitely complete category $\C$, we say
that~$Y$ is
\defn{protomodular} if every point with codomain $Y$ is stably
strong.

We write $\P(\C)$ for the full subcategory of $\C$ determined by
the protomodular objects.
\end{definition}

Obviously, every protomodular object is strongly unital. Hence it
is also unital and subtractive (Proposition~\ref{SU=SU}). We also
have:

\begin{proposition}\label{PM then Mal}
Let $\C$ be a finitely complete category. Every protomodular
object is a Mal'tsev object.
\end{proposition}
\begin{proof}
Let $Y$ be a protomodular object and consider the following
pullback of split epimorphisms:
\[
\xymatrix@!0@=5em{ A\times_{Y}C \splitsplitpullback
\ar@<-.5ex>[d]_{\pi_A} \ar@<-.5ex>[r]_(.7){\pi_C} & C
\ar@<-.5ex>[d]_g
\ar@<-.5ex>[l]_-{\langle sg,1_C \rangle} \\
A \ar@<-.5ex>[u]_(.4){\langle 1_A,tf \rangle} \ar@<-.5ex>[r]_f &
Y. \ar@<-.5ex>[l]_s \ar@<-.5ex>[u]_t }
\]
Since $Y$ is protomodular, the point $(g, t)$ is stably strong
and, consequently, also $(\pi_A, \langle 1_A,tf\rangle)$ is a
strong point. Moreover, the pullback of $s$ along $\pi_A$ is
precisely $\langle sg, 1_C\rangle$, so that the pair $(\langle
1_A,tf\rangle, \langle sg, 1_C\rangle)$ is jointly strongly
epimorphic, as desired. Observe that this proof is a simplified
version of that of Theorem~3.2.1 in~\cite{S-proto}.
\end{proof}

Note that, in the regular case, the above result follows from
Proposition~\ref{double split epi Mal'tsev} via Lemma~\ref{Lemma
Double}.

The inclusion $\P(\C)\subset \M(\C)$ is strict, in general, by the
following proposition, Proposition~\ref{Mal'tsev objects Mal'tsev
cat} and the fact that there exist Mal'tsev categories which are
not protomodular.

\begin{proposition}\label{proto objects proto cat}
If $\C$ is a finitely complete category, then $\P(\C)=\C$ if and
only if $\C$ is protomodular.
\end{proposition}
\begin{proof}
By definition, a finitely complete category is protomodular if and
only if all points in it are strong. When this happens,
automatically all of them are stably strong.
\end{proof}

\begin{corollary}
If $\C$ is a finitely complete category and $\P(\C)$ is closed
under finite limits in $\C$, then $\P(\C)$ is a protomodular
category.
\end{corollary}
\begin{proof}
Apply Proposition~\ref{proto objects proto cat} to $\P(\C)$.
\end{proof}

Observe that this hypothesis is satisfied when $\C$ is the
category $\Mon$ of monoids, or the category $\SRng$ of semirings,
as can be seen as a consequence of Example~\ref{semirings PM} and
Theorem~\ref{groups = protomodular monoids} below.

\begin{proposition}
If $\C$ is regular, then $\P(\C)$ is closed under quotients in
$\C$.
\end{proposition}
\begin{proof}
This follows immediately from Proposition~\ref{stably strong
points closed under quotients}.
\end{proof}

\begin{example} $\P(\SRng)= \M(\SRng)=\SU(\SRng)=\S(\SRng)=\Rng$. \label{semirings PM}
If $X$ is a protomodular semi{\-}ring, then it is obviously a
strongly unital semiring, thus a ring by
Theorem~\ref{SU(SRng)=Rng}. We already mentioned that if $X$ is a
ring, then every point over it in $\SRng$ is stably strong, since
it is a Schreier point by~\cite[Proposition~6.1.6]{SchreierBook}.
In particular, the category $\P(\SRng)=\Rng$ is closed under
finite limits and it is protomodular. Thanks to Propositions
\ref{PM then Mal} and~\ref{Mal'tsev implies strongly unital}, we
also have that $\M(\SRng) = \Rng$.
\end{example}

\begin{theorem} \label{groups = protomodular monoids}
If $\C$ is the category $\Mon$ of monoids, then
$\P(\C)=\M(\C)=\Gp$, the category of groups. In other words, the
following conditions are equivalent, for any monoid $M$:
\begin{tfae}
\item $M$ is a group; \item $M$ is a Mal'tsev object, i.e.,
$\Pt_{M}(\Mon)$ is a unital category; \item $M$ is a protomodular
object, i.e., all points over $M$ in the category of monoids are
stably strong.
\end{tfae}
\end{theorem}
\begin{proof}
If $M$ is a group, then every point over it is stably strong: by
Proposition~3.4 in \cite{BM-FMS2} it is a Schreier point, and
Schreier points are stably strong by Lemma 2.1.6 and Proposition
2.3.4 in \cite{SchreierBook}. This proves that (i) implies (iii).
(iii)~implies (ii) by Proposition~\ref{PM then Mal}, and (ii)
implies (i) by Theorem~\ref{Mal'tsev monoids are groups}.
\end{proof}

\begin{remark}
Note that, in particular, $\P(\Mon)$ is closed under finite limits
in the category $\Mon$.
\end{remark}

\begin{remark}
The proof of Theorem~\ref{Mal'tsev monoids are groups} may be
simplified to obtain a direct proof that (iii) implies (i) in
Theorem~\ref{groups = protomodular monoids}. Instead of the
pullback diagram~\eqref{M(Mon)}, we may consider the simpler
pullback of $\links 1_{M}\;1_M\rechts\colon M+M\to M$ along
$m\colon{\N\to M}$. This idea is further simplified and at the same time strengthened in the article~\cite{GM-ACS}.
\end{remark}

\begin{remark}
As recalled in Example~\ref{Greg}, there are gregarious monoids
that are not groups. Hence, in $\Mon$, the subcategory $\P(\Mon)$
is strictly contained in $\SU(\Mon)$.
\end{remark}

\begin{example}
In the category $\Cat_{X}(\C)$ of internal categories over a fixed
base object~$X$ in a finitely complete category $\C$, any internal
groupoid over~$X$ is a protomodular object. This follows from
results in~\cite{Bourn2014}: any pullback of any split epimorphism
over such an internal groupoid ``has a fibrant splitting'', which
implies that it is a strong point. So, over a given internal
groupoid over~$X$, all points are stably strong, which means that
this internal groupoid is a protomodular object.
\end{example}

Similarly to the Mal'tsev case, we also have:

\begin{proposition}
If $\C$ is a pointed finitely complete category, then the zero
object is protomodular if and only if $\C$ is unital.
\end{proposition}

\begin{proof}
The zero object $0$ is protomodular if and only if every point
over it is stably strong. This means that, for any $X$, $Y \in
\C$, in the diagram
\[
\xymatrix@!0@=5em{ X \ar@{{ |>}->}[r]^-{\langle 1_{X},0 \rangle}
\ar@{=}[dr] & X \times Y \halfsplitpullback
\ar@<-.5ex>[r]_(.7){\pi_Y}
\ar[d]^(.5){\pi_X} & Y \ar@<-.5ex>[l]_-{\langle 0,1_{Y} \rangle} \ar[d] \\
& X \ar@<-.5ex>[r] & 0, \ar@<-.5ex>[l] }
\]
the morphisms $\langle 1_{X}, 0 \rangle$ and $\langle 0, 1_{Y}
\rangle$ are jointly strongly epimorphic. This happens if and only
if $\C$ is unital.
\end{proof}

\begin{proposition}\label{PM via sum}
If $\C$ is a regular category with binary coproducts, then the
following conditions are equivalent:
\begin{tfae}
\item $Y$ is a protomodular object; \item for every morphism
$f\colon {X\to Y}$, the point
\[
(\links f\;1_{Y}\rechts\colon{X+Y\to Y},\quad \iota_{Y}\colon
{Y\to X+Y})
\]
is stably strong.
\end{tfae}
\end{proposition}
\begin{proof}
This follows from Proposition~\ref{stably strong points closed
under quotients} applied to the morphism of points
\[
\xymatrix@!0@=5em{ X+Y \ar@{->>}[r]^-{\links 1_{X}\; s\rechts}
\ar@<-.5ex>[d]_-{\links f\; 1_{Y}\rechts} & X
\ar@<-.5ex>[d]_f \\
Y \ar@{=}[r] \ar@<-.5ex>[u]_{\iota_{Y}} & Y, \ar@<-.5ex>[u]_s
 }
\]
for any given point $(f\colon X\to Y,s\colon Y\to X)$.
\end{proof}

\subsection{$\P(\C)$ is a protomodular core}
Similarly to what we did for Mal'tsev objects, we now show that
the subcategory $\P(\C)$ of protomodular objects is a protomodular
core with respect to a suitable class $\p$ of points, provided
that $\P(\C)$ is closed under finite limits in $\C$.

Let $\C$ be a finitely complete category such that $\P(\C)$ is
closed under finite limits. We define the class $\p$ in the
following way: a point $(f,s)$ belongs to~$\p$ if and only if it
is the pullback
\begin{equation*}
\label{diagram for PP} \vcenter{
\xymatrix@!0@=4em{ A \ophalfsplitpullback \ar[r] \ar@<-.5ex>[d]_-{f} & A' \ar@<-.5ex>[d]_{f'} \\
 X \ar[r] \ar@<-.5ex>[u]_-{s} & X' \ar@<-.5ex>[u]_{s'}
}}
\end{equation*}
of some point $(f',s')$ in $\P(\C)$. Note that $\p$ is a class of
strong points, since they are pullbacks of stably strong points
(the codomain $X'$ is a protomodular object). The class $\p$ is
also a pullback-stable class since any pullback of a point $(f,s)$
in~$\p$ is also a pullback of a point in $\P(\C)$. The class $\p$
is not closed under finite limits in~$\Pt(\C)$, in general. So,
strictly speaking, it does not give rise to an $\s$-protomodular
category. However, as we observed for the Mal'tsev case, the fact
(which follows immediately from the definition of $\p$) that all
points in $\P(\C)$ belong to $\p$ allows us to apply the same
arguments as in the proof of Proposition~6.2 in~\cite{S-proto}
(and its generalisation to the non-pointed case, given in
\cite{Bourn2014}) to conclude that $\P(\C)$ is a protomodular
category. Indeed, as we now show, it is the protomodular core
$\p(\C)$ of $\C$ relative to the class of points $\p$. In other
words, it is the category of $\p$-special objects of $\C$.

\begin{proposition}
\label{proto objs=proto core} If $\C$ is a pointed finitely
complete category, and the subcategory $\P(\C)$ of protomodular
objects is closed under finite limits in $\C$, then it coincides
with the protomodular core $\p(\C)$ consisting of the $\p$-special
objects of~$\C$.
\end{proposition}

\begin{proof}
If $X$ is a protomodular object, it is obviously $\p$-special,
since the point
\begin{equation*}
(\pi_{1}\colon{X\times X\to X},\quad \Delta_{X}=\langle
1_{X},1_{X} \rangle\colon {X\to X\times X})
\end{equation*}
belongs to the subcategory $\P(\C)$, which is closed under binary
products.

Conversely, suppose that $X$ is $\p$-special. Then the point
$(\pi_1,\Delta_X)$ is a pullback of a point $(f',s')$ in $\P(\C)$
\[
\xymatrix@!0@=5em{ X\times X \ophalfsplitpullback \ar[r]^-{h'} \ar@<-.5ex>[d]_{\pi_1} & A' \ar@<-.5ex>[d]_{f'} \\
 X \ar@<-.5ex>[u]_(.4){\langle 1_{X},1_{X} \rangle} \ar[r]_-h & B'. \ar@<-.5ex>[u]_{s'} }
\]
But then $X$, which is the kernel of~$\pi_2$, is also the kernel
of $f'$, and hence it belongs to $\P(\C)$.
\end{proof}

\section*{Acknowledgements}
We are grateful to Dominique Bourn for proposing the problem that
led to the present paper. We would also like to thank Alan Cigoli
and Xabier Garc\'ia-Mart\'inez for fruitful discussions, and the referee for careful  comments and suggestions on the text.

%
%

\providecommand{\noopsort}[1]{}
\providecommand{\bysame}{\leavevmode\hbox to3em{\hrulefill}\thinspace}
\providecommand{\MR}{\relax\ifhmode\unskip\space\fi MR }
\providecommand{\MRhref}[2]{%
  \href{http://www.ams.org/mathscinet-getitem?mr=#1}{#2}
}
\providecommand{\href}[2]{#2}

\end{document}